\documentclass[10 pt,a4 paper]{article}

\usepackage[T1]{fontenc}
\usepackage{lmodern}
\usepackage[english]{babel}

\usepackage{xcolor}
\usepackage{amsfonts}
\usepackage{amsmath}
\usepackage{amssymb}
\usepackage{amsthm}
\usepackage{array}

\usepackage{csquotes}
\usepackage[backend=bibtex, style=alphabetic,sorting=nyt,maxbibnames=9]{biblatex}
\usepackage{geometry}
\usepackage{graphicx}
\usepackage{subcaption}
\usepackage{tikz}
\usepackage[linktoc=page]{hyperref}
\hypersetup{colorlinks=true,linkcolor=blue}

\usepackage{blindtext}
\usepackage{bbm}
\usepackage{bm}
\usepackage{bold-extra}

\usepackage{calc}
\usepackage{calligra}
\usepackage{caption}
\usepackage{centernot}

\usepackage{dsfont}
\usepackage{dynkin-diagrams}

\usepackage{faktor}
\usepackage{fancybox}
\usepackage{fancyhdr}
\usepackage{float}

\usepackage{guit}
\usepackage{hhline}

\usepackage{listings}

\usepackage{mathdots}
\usepackage{mathrsfs}
\usepackage{mathtools}
\usepackage{multirow}

\usepackage{newlfont}

\usepackage{textcomp}

\usepackage{tikz-cd}
\usepackage{tikz-3dplot}
\tdplotsetmaincoords{70}{120} 
\tdplotsetrotatedcoords{0}{0}{0} 
\usepackage{tocloft}

\usepackage{verbatim}

\usepackage[all]{xy}

\usepackage{ytableau}

\usepackage{wrapfig}
\usepackage{wasysym}

\usetikzlibrary{matrix,arrows,decorations.pathmorphing, calc}

\newcolumntype{P}[1]{>{\centering\arraybackslash}p{#1}}

\newtheorem{prop}{Proposition}[section]
\newtheorem{lemma}[prop]{Lemma}
\newtheorem{cor}[prop]{Corollary}
\newtheorem{teo}[prop]{Theorem}
\newtheorem*{main-teo}{Main Theorem}

\newtheorem*{Weyl's dimension}{Weyl's dimension formula}
\newtheorem*{Borel-Weil}{Borel-Weil's Theorem}

\newtheorem*{prop*}{Proposition}
\newtheorem*{lemma*}{Lemma}
\newtheorem*{teo*}{Theorem}
\newtheorem*{cor*}{Corollary}

\theoremstyle{definition}
\newtheorem{Def}[prop]{\normalfont \scshape Definition}
\newtheorem{es}[prop]{\normalfont \scshape Example}
\newtheorem{obs}[prop]{\normalfont \scshape Remark}

\newtheorem*{Def*}{\normalfont \scshape Definition}
\newtheorem*{obs*}{\normalfont \scshape Remark}
\newtheorem*{es*}{\normalfont \scshape Example}

\numberwithin{equation}{section}
\numberwithin{figure}{section}
\numberwithin{table}{section}

\newcounter{daggerfootnote}

\DeclareMathOperator{\Ker}{ker}
\DeclareMathOperator{\im}{Im}

\DeclareMathOperator{\Sing}{Sing}

\DeclareMathOperator{\Pic}{Pic}

\DeclareMathOperator{\Hilb}{Hilb}

\DeclareMathOperator{\Sym}{Sym}
\DeclareMathOperator{\Gr}{Gr}

\DeclareMathOperator{\IG}{IG}
\DeclareMathOperator{\LG}{LG}
\DeclareMathOperator{\Q}{Q}
\DeclareMathOperator{\OP}{\mathbb O \mathbb P}

\DeclareMathOperator{\rk}{Rk}

\DeclareMathOperator{\expdim}{expdim}
\DeclareMathOperator{\Terr}{Terr}

\DeclareMathOperator{\Mat}{Mat}
\DeclareMathOperator{\GL}{GL}
\DeclareMathOperator{\SL}{SL}

\DeclareMathOperator{\Sp}{Sp}

\DeclareMathOperator{\Spin}{Spin}

\DeclareMathOperator{\Orth}{Orth}

\DeclareMathOperator{\Stab}{stab}

\DeclareMathOperator{\diag}{diag}

\DeclareMathOperator{\supp}{supp}

\definecolor{coloreteorema}{HTML}{E8EDFA}
{\endMakeFramed}

{\endMakeFramed}

\newcommand{\longdashleftrightarrow}[1][1.75em]{\mathrel{
    \tikz[baseline]{\draw[dash pattern=on .25em off .1 em,<->](0,.58ex)--(#1,.58ex)}}}

\title{\textsc{\Large Secant varieties of generalised Grassmannians}}
\author{\textsc{Vincenzo Galgano}\\ 
	\textsc{\footnotesize Max-Planck Institute of Molecular Cell Biology and Genetics, Dresden, Germany}\\
	\textit{\small E-mail address}: \texttt{\small galgano@mpi-cbg.de} ; {\small ORCID: \texttt{0000-0001-8778-575X}}
} 
\date{}

\begin{document}

\maketitle

\begin{abstract}
	\noindent Secant varieties of a homogeneously embedded generalised Grassmannian $G/P$ inherit the natural group action, and one can reduce the study of their local geometric properties to $G$--orbit representatives. The case of secant varieties of lines is particularly elegant as their $G$--orbits are induced by $P$--orbits in both $G/P$ and $\mathfrak{g}/\mathfrak{p}$. Parabolic orbits are a classical problem in Representation Theory, well understood when $G/P$ is cominuscule. Exploiting them, we provide a complete and uniform description of both the identifiable and singular loci of the secant variety of lines to any cominuscule variety. We also introduce a finer version of the $2$-nd Terracini locus, called $2$-nd strong-Terracini locus, and we determine it for cominuscule varieties. Finally, we analyse the non-cominuscule case of isotropic Grassmannians for comparison, and we highlight a few differences. \\
	\hfill\break
	{\bf Keywords:} flag variety, cominuscule variety, parabolic orbit, secant variety, identifiability, decomposition locus, singular locus, Terracini locus. \\
    {\bf MSC codes (2020):} 14L30, 14M15, 14N07. 	
\end{abstract}

\tableofcontents

\paragraph*{Acknowledgement.} {\em The author is grateful to Alessandra Bernardi and Giorgio Ottaviani for their motivating support in generalising results achieved in the author's PhD thesis. Sec.\ \ref{sec:isotropicgrassmannian} is the outcome of a visit at the Institut de Mathématiques de Toulouse {\em supported by the italian “National Group for Algebraic and Geometric Structures, and their Applications” (GNSAGA-INdAM)}: special thanks are due to Laurent Manivel for his inspiring mentorship and hospitality. The author also thanks Fulvio Gesmundo, Edoardo Ballico, Luca Chiantini and Martina Lanini for their clarifying discussions.}

\section*{Introduction}

This work provides a complete and uniform description of the identifiable locus and singular locus of secant varieties of lines to any cominuscule variety, extending and unifying results obtained for Grassmannians and Spinor varieties in \cite{galganostaffolani2022grass, galgano2023spinor}. Combined with the more algebraic analysis due to Manivel--Micha\l{}ek in \cite{manivelmichaleksecants}, our results aim to give a richer understanding of the geometry of secant (and tangential) varieties to cominuscule varieties. Moreover, we introduce a finer version of the classical Terracini locus, called {\em strong-Terracini locus}. \\
\hfill\break
\indent Let $X\subset \mathbb P^M$ be an irreducible linearly non-degenerate complex projective variety. Given a point $f \in \mathbb P^M$, the {\em $X$--rank} of $f$ with respect to $X$ is the minimum number of points of $X$ whose linear span contains $f$: in particular, points of $X$ have $X$--rank equal to $1$. For any $r\in \mathbb Z_{>0}$, the {\em $r$--th secant variety} $\sigma_r(X)\subset \mathbb P^M$ is the Zariski closure of points in $\mathbb P^M$ of $X$--rank at most $r$. Such varieties are central in {\em Tensor Decomposition} \cite{landsberg, bernardi2018hitchhiker}. The versatility of tensors for parameterising data makes this topic of great impact to life sciences. \\
\indent Two crucial aspects of secant varieties are the {\em identifiability} and the {\em singularity} of their points. A point $f\in \mathbb P^M$ of $X$--rank $r$ is identifiable if there exists a unique $r$--tuple of points of $X$ on whose linear span $f$ lies. However, one can consider other types of identifiability, such as the tangential-identifiability for tangent points \cite{galganostaffolani2022grass}. Identifiability and smoothness are quite related to each other and often, but not always, one notion suggests the other. The most recent result on this interplay is due to Yoon--Seo \cite{yoonseo}, stating that second secant varieties are always smooth along their identifiable locus. This plays a key role in Sec.\ \ref{subsec:sing cominuscule} of our paper. \\
\indent The problem of studying identifiability of tensors is an active area of research, also motivated by its importance to applications: it corresponds to uniqueness in recovering data. Several techniques have been introduced for detecting identifiability, such as several apolarities \cite{LO13, arrondo21skew, schurapolarity}, but still describing the whole set of identifiabile points in a secant variety is a hard task: such set is called the {\em identifiable locus}. Even when a tensor is not identifiable, one is interested in determining its {\em decomposition locus}, corresponding to the set of all minimal-length $0$-dimensional subschemes of $X$ whose linear span contains the given tensor: for most recent results we suggest \cite{bernardi2018new, bernardisantarsiero2024}. A variety strongly related to both identifiability and singularity is the {\em Terracini locus}, introduced by \cite{ballico2021terracini}. In the recent years many literature has been devoted to the Terracini locus \cite{ballico2020terracini, galuppi2023geometry, laface2024ample, ballico2024minimal, ballico2024terracini}. \\
\indent Determining the {\em singular locus} $\Sing(\sigma_r(X))$ of a secant variety $\sigma_r(X)$ is a classical problem. In general, if $\sigma_r(X)$ is not a linear space, the singular locus $\Sing(\sigma_r(X))$ contains $\sigma_{r-1}(X)$, but only in few cases $\Sing(\sigma_r(X))$ is actually determined. The most investigated varieties have been Segre varieties \cite{michalek2015secant}, Veronese varieties \cite{kanev1999chordal, kangjin2018, kangjin2021}, Grassmannians \cite{manivelmichaleksecants, galganostaffolani2022grass} and Spinor varieties \cite{galgano2023spinor}. These are flag varieties, and their underlying Representation Theory allows a rich description of their geometry, as in \cite{LM03, LM04}. \\
\indent The case $r=2$ has been widely investigated in the literature but, despite being the smallest case, is not completely understood in general. Elegant and very important results have been obtained under certain assumptions on the embedding of the variety ($3$--very ampleness), cf. \cite{vermeire2001some, vermeire2009singularities, ullery, olano2024singularities}. However, varieties containing lines (such as minimally embedded flag varieties) do not satisfy such assumptions. \\
\hfill\break 
\indent We focus on the $2$-nd secant variety (a.k.a. {\em secant varieties of lines}) to a (projective) {\em generalised Grassmannian} minimally embedded. The latter is a flag varieties $G/P_k$ where $P_k$ is the maximal parabolic subgroup of $G$ defined by the $k$--th simple root $\alpha_k$. We consider $G/P_k$ to be minimally embedded in the (projectivized) irreducible $G$--module $\mathbb P(V^G_{\omega_k})$ defined by the $k$--th fundamental weight $\omega_k$. All this objects are recalled in the preliminary Sec.\ \ref{sec:preliminaries}. \\
\indent The secant variety of lines to $G/P_k$ is the Zariski closure of the union of all secant lines to $G/P_k$ in $\mathbb P(V_{\omega_k}^G)$ and one can distinguish two subsets: the dense subset $\sigma_2^\circ(G/P_k)$ of points lying on bisecant lines, and the tangential variety $\tau(G/P_k):=\cup_{x \in G/P_k}T_x(G/P_k)$ given by the union of all tangent lines. In this case, we consider the {\em identifiable locus} to be given by the points lying on a unique either bisecant or tangent line, or equivalently the ones lying on the linear span of a unique $0$--dimensional subscheme of $G/P_k$ of length $2$ (either reduced or non-reduced). \\
\indent Starting with Zak's key work \cite{zak1993tangents} and continuing with \cite{Kaj99, baur2007secant, landsbergweymantangential,Angelini11,  boralevi2011secants, boralevi2013note, landsbergweymansecant,  manivelmichaleksecants, russo2016, blomenhofer2023nondefectivity}, a wide literature has been devoted to secant varieties of lines (as well as higher secants) to flag varieties, and more generally to $G$--varieties: the secant varieties inherit the group action, and results are achieved by exploiting this. Our work is in this spirit. The $G$--action on the subsets $\sigma_2^\circ(G/P_k)$ and $\tau(G/P_k)$ is governed respectively by the $P_k$--action on the flag variety $G/P_k$ itself and on its tangent space $\mathfrak{g}/\mathfrak{p}_k$.\\
\indent Studying parabolic orbits in $G/P_k$ and $\mathfrak{g}/\mathfrak{p}_k$ is a classical problem in the theory of algebraic groups and their representations \cite{hillerohrle, pech2014quantum}. The generalised Grassmannians for which this problem is better understood are the {\em cominuscule varieties}, characterised by the unipotent radical $P_k^u$ of the parabolic subgroup being abelian, or equivalently by the irreducibility of the nilpotent algebra $\mathfrak{g}/\mathfrak{p}_k\simeq \mathfrak{p}_k^u$ as $\mathfrak{p}_k$--module. The parabolic orbits in a cominuscule variety $G/P_k$ are deeply studied by Richardson--R\"{o}hrle--Steinberg in \cite{richardson}. On the other hand, in the cominuscule case, the study of the parabolic orbits in the Lie algebra $\mathfrak{g}/\mathfrak{p}_k$ is simplified by the fact that the unipotent radical acts trivially (being abelian), and the only action to be considered is the one of the Levi component of $P_k$.  \\
\indent In light of the above, our starting point is the study of the $G$--orbit poset in the secant variety of lines $\sigma_2(G/P_k)$ to a cominuscule variety, and we exploit it for studying the identifiable and singular locus of $\sigma_2(G/P_k)$. We also introduce the $2$-nd {\em strong-Terracini locus} $\mathbb{T}_2^{str}(G/P_k)$: this finer version of the classical Terracini locus measures \textquotedblleft how much\textquotedblright \space a variety is defective. We collect the main results at once as follows.

\begin{main-teo}\label{main-thm}
	Let $G/P_k \subset \mathbb P(V_{\omega_k}^G)$ be a cominuscule variety minimally embedded, and let $\sigma_2(G/P_k)$ be its secant variety of lines. Let $d_k^{_G}$ be as in \autoref{table:secants to cominuscule}. Let $\Sigma_2$ be the $G$--orbit of points in $\sigma_2(G/P_k)$ lying on both a bisecant line and a tangent line.
    \begin{enumerate}
        \item (Theorem \ref{thm:poset cominuscule}) The $G$--orbit poset graph of the secant variety of lines $\sigma_2(G/P_k)$ is
        
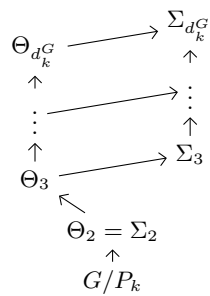
\begin{figure}[H]
	    \begin{center}
		{\small \begin{tikzpicture}[scale=1.8]
			
			\node(S) at (0,0.2){{$G/P_k$}};
			\node(t2) at (0,0.7){{$\Theta_2=\Sigma_2$}};
			\node(t3) at (-0.6,1.1){{$\Theta_3$}};
			\node(t) at (-0.6,1.6){{$\vdots$}};
			\node(td) at (-0.6,2.1){$\Theta_{d_{k}^{_G}}$};
			
			\node(s3) at (0.6,1.3){{$\Sigma_3$}};
			\node(s) at (0.6,1.8){{$\vdots$}};
			\node(sd) at (0.6,2.3){{$\Sigma_{d_{k}^{_G}}$}};
			
			\path[font=\scriptsize,>= angle 90]
			(S) edge [->] node [left] {} (t2)
			(t2) edge [->] node [left] {} (t3)
			(t3) edge [->] node [left] {} (s3)
			(t3) edge [->] node [left] {} (t)
			(t) edge [->] node [left] {} (td)
			(t) edge [->] node [left] {} (s)
			(td) edge [->] node [left] {} (sd)
			(s3) edge [->] node [left] {} (s)
			(s) edge [->] node [left] {} (sd);
			\end{tikzpicture}}
	       \end{center}
            \end{figure}
            where the orbits $\Sigma_j\subset \sigma_2^\circ(G/P_k)$ and $\Theta_j\subset\tau(G/P_k)$ are induced by the parabolic orbits in $G/P_k$ and $\mathfrak{g}/\mathfrak{p}_k$, as defined in Lemma \ref{lemma:secant orbits} and Corollary \ref{cor:tangent orbits}. 
        \item (Theorem \ref{thm:identifiability}) The unidentifiable locus in $\sigma_2(G/P_k)$ is $\Sigma_2$, and the dimensions of the decomposition loci of its points are as in \autoref{table:decomposition loci}.
        \item (Theorem \ref{thm:singular locus}) Whenever $\overline{\Sigma_2} \subsetneq \sigma_2(G/P_k) \subsetneq \mathbb P(V_{\omega}^G)$, the singular locus of $\sigma_2(G/P_k)$ is $\overline{\Sigma_2}=G/P_k \sqcup \Sigma_2$ and it has dimension as in \autoref{table:dim sing loci}.
        \item (Theorem \ref{thm:terracini locus cominuscule}) Whenever $\overline{\Sigma_2}\subsetneq \sigma_2(G/P_k)$, the $2$-nd strong-Terracini locus $\mathbb{T}_2^{str}(G/P_k)$ corresponds to $\overline{\Sigma_2}=G/P_k \sqcup \Sigma_2$. If $\overline{\Sigma_2}=\sigma_2(G/P_k)$, then $\mathbb{T}_2^{str}(G/P_k)$ is empty for $G/P_k=\Q_m$ quadric, and it corresponds to $G/P_k$ otherwise -- cf. \autoref{table:terracini loci}.
        \end{enumerate}
    \end{main-teo}

Sec.\ \ref{sec:preliminaries} contains preliminaries about all topics covered in this paper, in order to make it accessible to a wider audience. In Sec.\ \ref{sec:secant_cominuscule}, first we analyse how the parabolic orbits in a cominuscule variety $G/P_k$ and its tangent space $\mathfrak{g}/\mathfrak{p}_k$ induce orbits in the subsets $\sigma_2^\circ(G/P_k)$ and $\tau(G/P_k)$ respectively, then we study the inclusions among the orbit closures. Sec.\ \ref{sec:identif_sing} is devoted to the description of the identifiable locus and singular locus of the secant variety of lines to a cominuscule variety, as well as the $2$-nd strong-Terracini locus of a cominuscule variety. Our analysis also includes the dimensions of the decomposition loci of non-identifiable points, and of the singular locus. Finally, in Sec.\ \ref{sec:isotropicgrassmannian} we show that non-cominuscule generalised Grassmannians have a different behaviour, by studying the case of isotropic Grassmannians $\Sp_{2N}/P_k=\IG(k,2N)$ for $2 \leq k \leq N-1$. We provide a detailed description of the poset of $P_k$--orbits in the tangent space $\mathfrak{sp}_{2N}/\mathfrak{p}_k$, together with their dimensions. We also deduce the $\Sp_{2N}$--orbit poset in the tangential variety $\tau(\IG(k,2N))$, and analyse the (tangential-)identifiable locus.

\section{Preliminaries}\label{sec:preliminaries}

We work over the complex field $\mathbb C$. For any $m \in \mathbb N$, we denote $[m]:=\{1,\ldots, m\}$. 

\subsection{Miscellanea on Lie groups and Lie algebras}

Let $G$ be a simply connected semisimple complex Lie group. Fix $T \subset B \subset G$ a maximal torus and a Borel subgroup respectively. We adopt the gothic notation $\mathfrak{g}$, $\mathfrak{t}$ and $\mathfrak{b}$ for the corresponding complex Lie algebras, as well as for any other Lie algebra. Given $\Phi\subset \mathfrak{t}^\vee\setminus \{0\}$ an irreducible root system for $G$, let $\Phi^+=\Phi(B)$ be the set of positive roots with respect to $B$, and let $\Phi^-=\Phi\setminus \Phi^+$ be its complement. Set $\Delta=\{\alpha_1,\ldots, \alpha_N\}\subset \Phi^+$ to be the set of (positive) simple roots. For any $\alpha \in \Phi$, $\mathfrak{g}_\alpha$ denotes the corresponding weight space. We fix the notation $(\cdot | \cdot)$ for the scalar product on $\langle \Phi\rangle_{\mathbb R}$, and $\langle \cdot |\cdot \rangle$ for the Cartan number on roots in $\Phi$, that is $\langle \beta|\alpha\rangle =2\frac{(\beta|\alpha)}{(\alpha|\alpha)}$ for any $\alpha, \beta \in \Phi$. Let $\omega_1,\ldots, \omega_N$ be the fundamental weights spanning the weight lattice $\Lambda$. We denote the set of dominant weight by $\Lambda^+$. \\
\indent For any $\alpha \in \Phi$ consider the simple reflection $s_{\alpha}(v)=v-2\frac{(v|\alpha)}{(\alpha|\alpha)}\alpha$ in $\mathbb R^N$. Let $\cal W_G:= N_G(T)/T= \langle s_\alpha \ | \ \alpha \in \Phi \rangle=\langle s_{\alpha_i} \ | \ i=1:N \rangle$ be the Weyl group of $G$. For any $w \in \cal W_G$, we denote the length of a reduced expression of $w$ by $\ell(w)$, and the longest element in $\cal W_G$ by $w_0$.

\paragraph*{Parabolic subgroups.} Given $U_\alpha<G$ the root subgroup such that $Lie(U_\alpha)=\mathfrak{g}_\alpha$, we fix the following notation for parabolic subgroups: $P_I:=\langle B, U_{-\alpha} \ | \ \alpha \notin I \rangle$ for any $I\subset \Delta$. When $I=\{\alpha_i\}$ we write $P_i:=P_{\{\alpha_i\}}$ for the $i$--th maximal parabolic subgroup. Set $\Phi(I):=\{ \alpha \in \Phi \ | \ \langle \alpha | \omega_i \rangle \geq 0 \ \forall i \in I\}$, $\Phi(I)^0:=\{ \alpha \in \Phi \ | \ \langle \alpha | \omega_i \rangle =0 \ \forall i \in I\}$ and $\Phi(I)^+:=\Phi(I)\setminus \Phi(I)^0$. \\
\indent At the Lie algebra level, one has the {\em Levi decomposition} $\mathfrak{p}_I = \mathfrak{l}_I\oplus \mathfrak{p}_I^u$, where $\mathfrak{l}_I=\mathfrak{t}\oplus \bigoplus_{\alpha \in \Phi(I)^0} \mathfrak{g}_\alpha$ is reductive and 
\[\mathfrak{p}_I^u=\bigoplus_{\alpha \in \Phi(I)^+} \mathfrak{g}_\alpha \]
is nilpotent (with respect to the adjoint action). As a reductive Lie algebra, $\mathfrak{l}_I$ has semisimple component $\mathfrak{s}_I$ whose root system is $\Phi(I)^0$: in this respect, we say that $\Phi(I)^0$ are the roots of the parabolic algebra $\mathfrak{p}_I$. Similarly, the weight lattice of $\mathfrak{p}_I$ (strictly formally, of $\mathfrak{s}_I$) is $\Lambda_I \ := \ \sum_{i\in I}\mathbb Z\omega_i\subset \Lambda$, and we denote the dominant weight sublattice by $\Lambda^+_I$. \\
\indent Going back to the group level, the splitting $\mathfrak{p}_I=\mathfrak{l}_I\oplus \mathfrak{p}_I^u$ corresponds to the {\em Levi decomposition} $P_I=L_IP^u_I$, where $L_I$ and $P^u_I$ are respectively reductive and unipotent subgroups. Finally, the Weyl group of $P_I$ is $\cal W_{P_I}:=N_{P_I}(T)/T=\langle s_{\alpha_i} \ | \ \alpha_i \notin I\rangle$.

\paragraph*{$\mathbb Z$-gradings on Lie algebras.} Fix a simple root $\alpha_k\in \Delta$ and its corresponding maximal parabolic subgroup $P=P_k$. For any $\beta \in \Phi$ let $m_k(\beta)$ be the coefficient of $\alpha_k$ in $\beta$. Then one considers on $\mathfrak{g}$ the $\mathbb Z$-grading $\mathfrak{g}=\bigoplus_{\mathbb Z}\mathfrak{g}_i$ defined by
\[
\mathfrak{g}_0=\mathfrak{t}\oplus \bigoplus_{\beta \in \langle \Delta\setminus \{\alpha_k\}\rangle_\mathbb Z}\mathfrak{g}_\beta \ \ \ , \ \ \ \mathfrak{g}_i = \bigoplus_{m_k(\beta)=i}\mathfrak{g}_{\beta} \ . \] 
In particular, $\mathfrak{p}=\bigoplus_{i\geq 0}\mathfrak{g}_i$ and $\mathfrak{g}/\mathfrak{p}\simeq \mathfrak{g}_{-1}\oplus \mathfrak{g}_{-2}\oplus \ldots\simeq \mathfrak{p}^u$, where $\mathfrak{p}^u$ is a $P$--module after adjoint action and the last identification of $P$--modules is given by the Killing form on $\mathfrak{g}$. Notice that only the summand $\mathfrak{g}_{-1}$ is a $\mathfrak{p}$-submodule of $\mathfrak{g}/\mathfrak{p}$. Moreover, given $\rho$ the longest root in $\Phi$, one has $m_k(\beta)\leq m_k(\rho)$. Thus, $\mathfrak{g}/\mathfrak{p}$ splits in as many summands as $m_k(\rho)$. We refer to the Bourbaki tables \cite[Planche I-IX]{bourbaki} for the coefficients of the simple roots in $\rho$.

\paragraph*{Irreducible representations.} We use the terms \textquotedblleft $G$-representation" and \textquotedblleft $G$-module" indistinctly. Since in our assumptions $G$ is simply connected, every $G$--module is a $\mathfrak{g}$--module and {\em viceversa}. Given $\mathfrak{g}$ of Dynkin type $\cal D$, we denote by $V_\lambda^{\cal D}$ (or also $V_\lambda^{\mathfrak{g}}$ or $V_\lambda^{G}$) the irreducible $\mathfrak{g}$--representation having highest weight $\lambda \in \Lambda^+$, and by $v_{\lambda}\in V_\lambda^\cal D$ its (unique up to scalars) highest weight vector. We refer to an irreducible representation $V_{\omega_k}^\cal D$, for $\omega_k$ a fundamental weight, as to {\em fundamental representation}. Finally, given a subset of simple roots $I \subset \Delta$, for any $\lambda \in \Lambda_I^+$ we write $V_\lambda^{\mathfrak{p}_I}$ for the irreducible $\mathfrak{p}_I$--module obtained from the irreducible $\mathfrak{l}_I$--module $V_{\lambda}^{\mathfrak{l}_I}$ by letting the nilpotent component $\mathfrak{p}_I^u$ act trivially.

\subsection{Generalised Grassmannians}\label{sec:cominuscule}

Given $G$ a simply connected semisimple complex Lie group, let $V_{\omega_k}^G$ be the $k$--th fundamental $G$--representation with highest weight vector $v_{\omega_k}$. The (unique closed) orbit $G\cdot [v_{\omega_k}]\subset \mathbb P(V_{\omega_k}^G)$ is a projective homogeneous variety. The stabilizer of the line $[v_{\omega_k}]$ is the maximal parabolic subgroup $P_k=P_{\{\alpha_k\}}$, so that $G/P_k=G\cdot [v_{\omega_k}]$. For $G$ of Dynkin type $\cal D$, we use indistinctly the notation $G/P_k = \cal D/P_k$, and we refer to such varieties as to {\em generalised Grassmannians} (in their minimal embedding).

\paragraph*{Cominuscule varieties.} A fundamental weight $\omega_k$ is {\em cominuscule} if the longest root $\rho \in \Phi$ has coefficient $1$ on the simple root $\alpha_k$. Given $P_k$ the maximal parabolic subgroup corresponding to the fundamental weight $\omega_k$, the generalised Grassmannian $G/P_k$ is {\em cominuscule} if $\omega_k$ is so. Cominuscule (fundamental) weights (and cominuscule varieties) only appear in Dynkin types $ABCD$ and $E_6,E_7$, and they are:
\begin{itemize}
\item the {\em Grassmannian} $\Gr(k,N+1)=\SL_{N+1}/P_k$ of $k$--dimensional linear subspaces in $\mathbb C^{N+1}$;
\item the $m$--dimensional {\em quadric} $\Q_m=G/P_1\subset \mathbb P^{m+1}$ for $G$ of either $B$-- or $D$--type;
\item the {\em Lagrangian Grassmannian} $\LG_{N}:=\LG(N,2N)=\Sp_{2N}/P_N$ of maximal $\boldsymbol{\omega}$--isotropic linear subspaces of $\mathbb C^{2N}$ for a given non-degenerate symplectic form $\boldsymbol{\omega}\in \bigwedge^2(\mathbb C^{2N})^\vee$;
\item the {\em Spinor varieties} $\mathbb S_N^\pm = \Spin_{2N}/P_{N-1},\Spin_{2N}/P_N$ of maximal $\bold{q}$--isotropic linear subspaces of $\mathbb C^{2N}$ for a given non-degenerate quadratic form $\bold{q}\in \Sym^2(\mathbb C^{2N})^\vee$;
\item the $16$--dimensional {\em Cayley plane} $\OP^2=E_6/P_1\simeq E_6/P_6$;
\item the $27$--dimensional variety $E_7/P_7$.
\end{itemize}
\indent From the theory of $\mathbb Z$--gradings on $\mathfrak{g}$, the weight $\omega_k$ being cominuscule is equivalent to $\mathfrak{g}$ having $\mathbb Z$--grading with respect to $P_k$ of the form $\mathfrak{g}=\mathfrak{g}_{-1}\oplus \mathfrak{g}_0 \oplus \mathfrak{g}_1$, which again is equivalent to the tangent space $\mathfrak{g}/\mathfrak{p}_k\simeq \mathfrak{g}_{-1}$ being an irreducible $\mathfrak{p}_k$-module (and the tangent bundle $\cal T(G/P_k)$ an irreducible $P_k$--homogeneous bundle as well). Moreover, the above conditions are equivalent to the unipotent radical $P^u_k$ being abelian \cite[Lemma 2.2]{richardson}.

\begin{obs*}
	The Cayley plane $\OP^2$ is a Severi variety, while the varieties $\Gr(3,6)$, $\LG_3$, $\mathbb S_{6}^\pm$ and $E_7/P_7$ are {\em Legendrian varieties}. They all fit in the second and third row of the {\em Freudenthal's magic square}, whose geometry has been studied in detail by F. Zak in \cite[Sec.\ III.2.5]{zak1993tangents}, and by J.M. Landsberg and L. Manivel in \cite{LM01, LM07}.
 \end{obs*}

\subsection{Secant varieties, identifiability and Terracini loci}\label{subsec:secant}

Given $X \subset \mathbb P^M$ an irreducible non-degenerate projective variety, the $X$-{\it rank} of a point $p \in \mathbb P^M$, denoted by $r_X(p)$, is the minimum number of points of $X$ whose span contains $p$. The {\it $r$-th secant variety $\sigma_r(X)$ of $X$ in $\mathbb P^M$} is defined as the Zariski closure of the set $\sigma_r^\circ(X) := \{p \in \mathbb P^M: r_X(p) \leq r\} \subset \mathbb P^M$ of points of rank at most $r$, i.e.
\[\sigma_r(X) := \overline{\sigma_r^\circ(X)}=\overline{\{p \in \mathbb P^M: r_X(p) \leq r\}} \subset \mathbb P^M \ . \]
The {\it border $X$-rank} of a point $p \in \mathbb P^M$, denoted by $br_X(p)$, is the minimum integer $r$ such that $p \in \sigma_r (X)$. Although computing the dimension of secant varieties is a hard problem in general, the following inequality always holds:
\[\dim \sigma_r (X) \leq \expdim \sigma_r (X):=\min \{r(\dim X+1)-1, M\} \ , \]
where the right-hand side is called {\it expected dimension} of the secant variety and is a straightforward consequence of the celebrated Terracini's Lemma. If $\dim\sigma_r(X)\lneq \expdim \sigma_r(X)$, then the secant variety $\sigma_r(X)$ is said to be {\em defective}, and the value $\delta_r := \expdim \sigma_r(X)-\dim \sigma_r(X)$ is called {\em defect}. We say that $\sigma_r(X)$ {\em overfills} the ambient space if $\dim \sigma_r(X)=M \lneq r(\dim X+1)-1$, while it {\em perfectly fills} the ambient space if $\dim \sigma_r(X)=M=r(\dim X+1)-1$.

\paragraph*{Secant variety of lines.} For $r=2$, the $2$-nd secant variety $\sigma_2(X)\subset \mathbb P^M$ is also called {\em secant variety of lines} to $X$ in $\mathbb P^M$, being the union of all projective lines intersecting $X$ in two points (counted with multiplicity). More precisely, it consists of two (non necessarily disjoint) subsets: the dense subset $\sigma_2^\circ(X)=\bigcup_{x,y\in X}\langle x, y \rangle$ is the union of all bisecant lines to $X$; the tangential variety $\tau(X):=\bigcup_{x\in X}T_xX$ is the union of all tangent lines to $X$. Notice that both subsets contain $X$ itself. Moreover, from Zak's key result \cite[Theorem 1.4]{zak1993tangents} it holds that either $\dim\tau(X) = 2\dim X$ {\em and} $\dim \sigma_2(X)=2\dim X+1$, or $\tau(X)=\sigma_2(X)$. 

\paragraph*{Abstract secant variety.} Let $\Hilb_2(X)$ be the Hilbert scheme of $2$ points of $X$. The {\em $2$-nd abstract secant variety} of $X$ is the incidence variety
\begin{equation}\label{def:smooth abstract secant}
\widetilde{Ab\sigma_2}(X):=\left\{(\cal Z,p)\in \Hilb_2(X)\times \mathbb P^M \ | \ p \in \langle \cal Z \rangle\right\} 
\end{equation}
endowed with the two natural projections $\widetilde{\pi_H}$ and $\widetilde{\pi_\sigma}$ onto $\Hilb_2(X)$ and $\mathbb P^M$ respectively. The image of $\widetilde{\pi_\sigma}$ is the secant variety of lines. Under the above definition, $\widetilde{Ab\sigma_2}(X)$ is smooth. An alternative, more classical, definition replaces $\Hilb_2(X)$ by the symmetric square $S^2X := (X\times X)/\mathfrak{S}_2$, and we denote it $Ab\sigma_2(X)$: this has singularities along the diagonal $\Delta_X$, and the image of the projection $\pi_\sigma: Ab\sigma_2(X)\rightarrow \mathbb P^M$ is the dense subset $\sigma_2^\circ(X)$. We will use both definitions (cf. proof of Proposition \ref{prop:singular LG} and Sec.\ \ref{subsec:terracini locus}). Note that in both definitions the $2$-nd abstract secant variety has dimension $2\dim X+1$.

\paragraph*{Identifiability.} A point $p \in \sigma_2(X)\setminus X $ is {\em identifiable} if it lies on a unique bisecant line to $X$, that is the fiber $\widetilde{\pi_\sigma}^{-1}(p)\subset \widetilde{Ab\sigma_2}(X)$ is a singleton corresponding to a unique reduced point of $\Hilb_2(X)$. A point $q \in \tau(X)\setminus X$ is {\em tangential-identifiable} if it lies on a unique tangent space to $X$, that is the fiber $\widetilde{\pi_\sigma}^{-1}(q)\subset \widetilde{Ab\sigma_2}(X)$ is a singleton corresponding to a unique non-reduced point of $\Hilb_2(X)$. We call {\em identifiable locus} of $\sigma_2(X)$ the subset of all identifiable and tangential-identifiable points. The {\em decomposition locus} of a point $p \in \sigma_2^\circ(X)\setminus X$ is the set $\widetilde{\pi_H} \circ \widetilde{\pi_\sigma}^{-1}(p)$ of all points in $\Hilb_2(X)$ whose linear span contains $p$.

\begin{obs}\label{rmk:defectivity-unidentifiability}
    If $\sigma_2(X)\subset \mathbb P^M$ is either defective or overfills the ambient space, then all points in $\sigma_2(X)\setminus X$ are unidentifiable, and the dimension of the decomposition locus of a generic point of $\sigma_2(X)$ coincides with the value $2\dim X + 1 - \dim \sigma_2(X)$. On the other end, if $\dim \sigma_2(X)= 2\dim X +1$ and a generic point is identifiable, then $\sigma_2(X)$ is said {\em generically identifiable}.
    \end{obs}

\paragraph*{Terracini locus.} 

The {\em $2$-nd Terracini locus}, classically introduced in \cite{ballico2021terracini}, is the locus of points in $S^2 X \setminus \Delta_X$ for which the differential of the projection $\pi_\sigma: Ab\sigma_2(X) \rightarrow \sigma_2^\circ(X)$ drops rank:
\[ \mathbb{T}_2(X):= \overline{\{ (x,y)\in S^2X\setminus \Delta_X \ | \ \dim \langle T_xX,T_yX \rangle \lneq 2\dim X + 1 \}} \ . \]

\noindent The above definition can be reinterpreted in terms of the projection $\widetilde{\pi_\sigma}: \widetilde{Ab\sigma_2}(X)\rightarrow \sigma_2(X)$, bu just replacing $S^2X\setminus \Delta_X$ with the locus $\Hilb_2^{red}(X)$ of reduced $0$--dimensional subschemes of $X$ of length $2$:

 \[ \widetilde{\mathbb{T}_2}(X) := \overline{\{ \{x,y\}\in \Hilb_2^{red}(X) \ | \ \dim \langle T_xX,T_yX \rangle \lneq 2\dim X + 1 \}} \ . \]

\begin{obs}\label{rmk:classical Terracini locus and Terracini lemma}
The (strict) upper-bound in the above definition is the dimension of the $2$-nd abstract secant variety. By Terracini's Lemma, for generic $(x,y)\in S^2X\setminus \Delta_X$ it holds $\dim \langle T_xX,T_yX\rangle =\dim\sigma_2(X)$. Then, if $\sigma_2(X)$ is not generically identifiable, one gets $\mathbb{T}_2(X)=S^2X\setminus \Delta_X$. This motivates us to introduce a finer version of the Terracini locus in Sec.\ \ref{subsec:terracini locus}.
\end{obs}

\begin{obs}\label{rmk:terracini is non-identifiable with positive fiber}
Terracini locus, identifiability and singularity are quite related. Non-identifiable points in $\sigma_2^\circ(X)$ with positive dimensional fiber along $\widetilde{\pi_\sigma}$ correspond to points in $\widetilde{\mathbb{T}_2}(X)$: cf. \cite[Prop.\ 6.2]{ballico2021terracini}. On the other hand, if $q \in \sigma_2(X)$ is smooth and with $0$--dimensional fiber along $\widetilde{\pi_\sigma}$, then Zariski's Main Theorem yields $\widetilde{\pi_H}\circ \widetilde{\pi_\sigma}^{-1}(q) \subset \Hilb_2(X)\setminus \widetilde{\mathbb{T}_2}(X)$. In particular, if $q \in \sigma_2(X)$ has finite decomposition locus lying in $\widetilde{\mathbb{T}_2}(X)$, then $q$ is singular.
\end{obs}

\subsection{Hamming distance on projective varieties}\label{subsec:hamming}

Given $X\subset \mathbb P^M$ an irreducible projective variety, the {\em Hamming distance} $d(p,q)$ between two points $p,q \in X$ is the minimum number of lines in $\mathbb P^M$ which are fully contained in $X$ and connect $p$ to $q$, that is $d(p,q) := \min \left\{ r \ | \ \exists p_1,\dots,p_{r-1} \in X \ \text{s.t.} \ L(p_i,p_{i+1})\subset X, \ i = 0:r-1 \right\}$ where $p_0=p$ and $p_r=q$ and $L(p_i,p_{i+1})$ is the line passing through $p_i$ and $p_{i+1}$. In particular, $d(p,q)=0$ if and only if $p=q$, while $d(p,q)=1$ if and only if the line $L(p,q)$ fully lies in $X$ and $p\neq q$. The {\em diameter} of $X$ is the maximum Hamming distance possible between its points.

\begin{obs}\label{rmk:hamming for G-varieties}
	Given $G$ a group, $W$ a $G$--representation and $X\subset \mathbb P(W)$ a variety which is invariant under the $G$-action induced on $\mathbb P(W)$, then the $G$-action preserves the Hamming distance between points in $X$, that is $d(g\cdot p, g\cdot q)=d(p,q)$.
	\end{obs}
 
\begin{obs*} 
    For cominuscule varieties the Hamming distance coincides with the minimum possible degree of rational curves passing through the two points \cite[Lemma 4.2]{buch2013}. For non-cominuscule varieties this is not true in general. Moreover, for generalised Grassmannians of classical type $ABCD$, this notion of Hamming distance corresponds to the notion of {\em subspace distance} in Coding Theory.
    \end{obs*}

\subsection{Secants of lines to generalised Grassmannians}
\label{subsec:sec2 for generalised}

In the previous notation, let $G/P_k=G\cdot [v_{\omega_k}]\subset \mathbb P(V_{\omega_k}^G)$ be a projective generalised Grassmannian in its minimal embedding. Our aim is to study the geometry of $\sigma_2(G/P_k)\subset \mathbb P(V_{\omega_k}^G)$. \\
\indent The $G$--action on $V_{\omega_k}^{G}$ preserves the $G/P_k$--rank of points in $\mathbb P(V_{\omega_k}^G)$, hence it leaves $\sigma_2(G/P_k)$ invariant. Moreover, the dense subset $\sigma_2^\circ(G/P_k)$ (of points lying on bisecant lines) and the tangential variety $\tau(G/P_k)$ (of tangent points) are $G$--invariant as well. Thus each one of these objects is union of $G$--orbits. In the following, we give a first partial description of the $G$--orbits in $\sigma_2^\circ(G/P_k)$ and $\tau(G/P_k)$ separately.

\paragraph*{$G$--orbits in $\sigma_2^\circ(G/P_k)$.} Points in $\sigma_2^\circ(G/P_k)$ are of the form $[p+q]\in \mathbb P(V_{\omega_k}^G)$, for $p,q$ in the affine cone $\widehat{G/P_k}$. Thus the $G$--orbits in the dense subset $\sigma_2^\circ(G/P_k)$ are determined by the $G$--orbits in $(G/P_k)\times (G/P_k)$, on which $G$ acts diagonally:
\[ g \cdot ([x],[y]) = (g\cdot [x],g\cdot [y]) \ \ \ \ \ , \ \forall g \in G, \ \forall [x],[y]\in G/P_k \ . \]

\begin{obs}\label{rmk:collapsing in secant branch}
    Despite $G$--orbits in $\sigma_2^\circ(G/P_k)$ are determined by the $G$--orbits in $(G/P_k)\times (G/P_k)$, they are not {\em a priori} in bijection. For instance, for any pair $([p],[q])$ of points having Hamming distance $1$, the point $[p+q]$ lies in the minimal orbit $G/P_k$ but the pair $([p],[q])$ does not lies in the diagonal orbit of $(G/P_k)\times (G/P_k)$. 
    \end{obs}

By $G$--homogeneity of $G/P_k$, any pair $([x],[y])\in (G/P_k)\times (G/P_k)$ is $G$--conjugated to a pair of the form $([v_{\omega_k}],[y'])$, whose first entry is left fixed by the action of $P_k$. Since $P_k$ moves the second entry only, the $G$--orbits in $(G/P_k)\times (G/P_k)$ are in bijection with the $P_k$--orbits in $G/P_k$: in particular, we get the correspondence between the cosets
\[ \frac{\sigma_2^\circ(G/P_k)}{G} \ \longdashleftrightarrow \ \frac{(G/P_k)\times (G/P_k)}{G} \ \stackrel{1:1}{\longleftrightarrow} \ \frac{G/P_k}{P_k}=P_k\!\setminus \!G/P_k \ , \]
where the first correspondence is dashed in light of Remark \ref{rmk:collapsing in secant branch}. From the Bruhat decomposition $G=\bigsqcup_{w \in \cal W_G} BwB$, one recovers the Bruhat decomposition $G/P_k=\bigsqcup_{w \in \cal W_G/\cal W_{P_k}} BwP_k$, so that 
\[ P_k\!\setminus \!G/P_k = \bigsqcup_{w \in \cal W_{P_k}\!\setminus \!\cal W_G/\cal W_{P_k}}P_kwP_k \ . \]

\begin{obs}\label{rmk:secant orbit finitely many}
	The $P_k$--orbits in $G/P_k$ (hence the $G$--orbits in $\sigma_2^\circ(G/P_k)$) are always finitely many, as parameterised by representatives of the Weyl group.
\end{obs}

\paragraph*{$G$--orbits in $\tau(G/P_k)$.} Again by $G$--homogeneity of $G/P_k$, the action of $G$ on the tangential variety $\tau(G/P_k)$ conjugates all tangent spaces $T_{[x]}(G/P_k)$ as $[x]\in G/P_k$ varies: indeed, $g\cdot T_{[x]}(G/P_k) = T_{[g\cdot x]}(G/P_k)$. It follows that any $G$--orbit $\cal O$ in $\tau(X)$ is of the form 
\[ \cal O=\bigcup_{[x]\in G/P_k}(\cal O \cap T_{[x]}(G/P_k)) = \bigcup_{g\in G} g\cdot (\cal O \cap T_{[v_{\omega_k}]}(G/P_k)) \ \]
hence it is enough to determine $\cal O\cap T_{[v_{\omega_k}]}(G/P_k)$ and then move it by $G$--action. But $\cal O \cap T_{[v_{\omega_k}]}(G/P_k)$ is a $P_k$--orbit, hence determining the $G$--orbits in $\tau(G/P_k)$ is equivalent to determining the $P_k$--orbits in $T_{[v_{\omega_k}]}(G/P_k)\simeq \mathfrak{g}/\mathfrak{p}_k\otimes \mathbb C v_{\omega_k}\simeq \mathfrak{p}_k^u$:
\[ \frac{\tau(G/P_k)}{G} \ \longleftrightarrow \ \frac{T_{[v_{\omega_k}]}(G/P_k)}{P_k} \ \longleftrightarrow \ \frac{\mathfrak{g}/\mathfrak{p}_k}{P_k} \ . \]
\begin{obs*}
	Unlike the orbits in $\sigma_2^\circ(G/P_k)$ (cf. Remark \ref{rmk:secant orbit finitely many}), the ones in the tangent space $\mathfrak{g}/\mathfrak{p}_k\simeq \mathfrak{p}^u_k$ may be infinitely many. The cases in which the nilpotent algebra $\mathfrak{p}_k^u$ has finitely many $P_k$-orbits have been classified by L. Hille and G. R\"{o}hrle \cite[Theorem 1.1]{hillerohrle}.
\end{obs*}

\section{Orbit poset in $\sigma_2(G/P_k)$ in the cominuscule case}\label{sec:secant_cominuscule}

In this section we assume the fundamental weight $\omega_k$ (hence $G/P_k$) to be cominuscule. 

\begin{obs}\label{rmk:exceptional}
    The geometry of $\OP^2$ and $E_7/P_7$ has been completely analysed in \cite[Secc.\ 4.1, 5.3]{LM01}. In particular, the orbit-closure posets of their secant varieties of lines are
    \[ \OP^2 \subset \sigma_2(\OP^2) \ \ \ \ \ , \ \ \ \ \ E_7/P_7 \subset \sigma_+(E_7/P_7) \subset \tau(E_7/P_7) \subset \sigma_2(E_7/P_7)=\mathbb P^{55} \ . \]
    Moreover, the geometry of quadrics is trivial as they are hypersurfaces. In this respect, the proofs in this section are mainly focus on the other cominuscule varieties of classical type.
    \end{obs}

\subsection{Orbits in $\sigma_2^\circ(G/P_k)$} 

From Remark \ref{rmk:secant orbit finitely many} we know that the dense subset $\sigma_2^\circ(G/P_k)$ has finitely many $G$--orbits parameterised by double cosets $\cal W_{P_k}\!\!\setminus \!\cal W/\cal W_{P_k}$ of the Weyl group. Their description straightforwardly follows from the description of $P_k$--orbits in a cominuscule $G/P_k$ due to R.\ Richardson, G.\ R\"{o}hrle and R.\ Steinberg \cite{richardson}. Before stating such result, we recall that the root system associated to the unipotent radical $P_k^u$ is 
\[\Phi_k^+:=\Phi(\{\alpha_k\})^+=\{ \alpha \in \Phi \ | \ \langle \alpha | \omega_k \rangle \gneq 0\} \ , \] 
and an ordered sequence of roots $(\beta_1,\ldots, \beta_\ell)\subset \Phi_k^+$ is orthogonal if $(\beta_i |\beta_j) =0$ for all $i\neq j$.
We denote by $d_k^{_G}$ the length of a maximal orthogonal sequence of long roots in $\Phi_k^+$.\\
\indent In the following table we list all the cominuscule varieties together with dimensions, minimal homogeneous embeddings, dimensions of their secant varieties of lines, and the values $d_{k}^{_G}$. For the Grassmannian $\Gr(k,N)$, the symbol $(\boldsymbol{\star})$ is to remark that: by duality of Grassmannians we assume $2k \leq N$; for $N\geq 6$ the secant variety $\sigma_2(\Gr(2,N))$ is defective of dimension $4N-11$. 

\begin{table}[h]
	\centering 
	{\small \begin{tabular}{ |P{2cm}|P{1.4cm}|P{3.2cm}|P{5cm}|P{0.8cm}|  }
			\hline
			$G/P_k$ & $\dim G/P_k$ & $\mathbb P(V_{\omega_k}^G)$ & $\dim \sigma_2(G/P_k)$ & $d_{k}^{_G}$ \\
			\hline
			\hline
			$\Gr(k,N)$ ($\boldsymbol{\star}$)& $k(N-k)$ & $\mathbb P(\bigwedge^k\mathbb C^{N})$ & $\min\{2k(N-k)+1, \binom{N}{k}-1\}$ & $k$  \\
			
			\hline
			$\Q_{2N-2}$ & $2N-2$ & $\mathbb P^{2N-1}$ & $2N-1$ ({\em overfills $\mathbb P^{2N-1}$}) & $2$ \\
			
			\hline
			$\LG_{N}$ & $\binom{N+1}{2}$ & $\mathbb P(V_{\omega_N}^G)\subsetneq\mathbb P(\bigwedge^N\mathbb C^{2N})$ & $\min\{N(N+1)+1, \binom{2N}{N}-\binom{2N}{N-2}-1\}$ & $N$ \\
			
			\hline
			$\mathbb S_N^\pm$ & $\binom{N}{2}$ & $\mathbb P(\bigwedge^{ev/od}\mathbb C^N)$ & $\min\{N(N-1)+1, 2^{N-1}-1\}$ & $\lfloor\frac{N}{2}\rfloor$ \\
			
			\hline
			$\OP^2$ & $16$ & $\mathbb P^{26}$ & $25$ ( {\em cubic hypers. , defective} ) & $2$  \\
			
			\hline
			$E_7/P_7$ & $27$ & $\mathbb P^{55}$ & $55$ ( {\em perfectly fills $\mathbb P^{55}$} ) & $3$ \\
			\hline
	\end{tabular} }
	\caption{Dimensions of secant varieties of lines to cominuscule varieties, and values $d_{k}^{_G}$.}
	\label{table:secants to cominuscule}
\end{table}

\begin{prop}[\cite{richardson}]\label{prop:secant orbit cominuscule}
	There are $d_{k}^{_G}+1$ many $P_k$--orbits in a cominuscule variety $G/P_k$, for $d_{k}^{_G}$ as in \autoref{table:secants to cominuscule}. More precisely, given $(\beta_1,\ldots, \beta_{d_k^{_G}})$ a maximal orthogonal sequence of long roots in $\Phi_k^+$, the $P_k$--orbits in $G/P_k$ are
		\[ G/P_k \ = \ P_k \ \sqcup \ \bigsqcup_{j=1}^{d_{k}^{_G}}P_kw_jP_k \ , \]
    where $w_j:= s_{\beta_1}\cdots s_{\beta_j}$ for any $j \in [d_k^{_G}]$. Moreover, such orbits are totally ordered, namely
	\[ \overline{P_kw_jP_k} = \bigsqcup_{i=1}^j P_kw_iP_k \ \ \ , \ \forall j\in [d_{k}^{_G}] \ . \]
\end{prop}
\begin{proof}
	The number of orbits appears in \cite[Sec.\ 2, Prop.\ 2.11 \& Table 1]{richardson}, while the total order of the orbits is proved in \cite[Corollary 3.7(a)]{richardson}.
\end{proof}

\indent From the above proposition we know that the direct product $G/P_k\times G/P_k$ splits in $d_{k}^{_G}+1$ many $G$--orbits, with representatives $(idP_k,idP_k)$ and $(idP_k, w_jP_k)$ for any $j \in [d_{k}^{_G}]$. Such pairs of representatives have a nice geometric property which one deduces by putting together \cite[Proposition 4.1(c)]{buch2013} and \cite[Lemma 4.2]{buch2013}: for any $j \in [d_k^{_G}]$ the pair $(idP_k,w_jP_k)$ has Hamming distance $j$. In particular, the point $[idP_k + w_1P_k]$ lies in $G/P_k$, and both orbits $G \cdot (idP_k,idP_k)$ and $G \cdot (idP_k,w_1P_k)$ in the direct product collapse to the minimal orbit $G/P_k$ in $\sigma_2^\circ(G/P_k)$.

\begin{obs*}
    In the above notation, we identify the points $[idP_k + w_jP_k]$ and $[v_{\omega_k}+w_j\cdot v_{\omega_k}]$ in $\sigma_2^\circ(G/P_k)$. Note that the point $w_j\cdot v_{\omega_k}$ is well-defined: indeed, the Weyl element $w_j\in \cal W_{G}=N_G(T)/T$ lifts to a coset $w_jT \subset G$, but the maximal torus $T\subset P_k$ acts trivially on $v_{\omega_k}$.
    \end{obs*}

\begin{lemma}\label{lemma:secant orbits}
    For $G/P_k$ cominuscule, the dense subset $\sigma_2^\circ(G/P_k)$ splits in the $d_{k}^{_G}$--many $G$--orbits
    \[ \Sigma_j \ := \ G\cdot [idP_k + w_jP_k] \ = \ G \cdot [v_{\omega_k} + w_j\cdot v_{\omega_k}] \ \ \ \ \ , \ \forall j\in [d_{k}^{_G}] \]
    where $w_j:=s_{\beta_1}\cdots s_{\beta_j}\in \cal W_G$ are defined by a maximal orthogonal sequence $(\beta_1,\ldots, \beta_{d_k^{_G}})$ of long roots in $\Phi_k^+$. In particular, $\Sigma_1=G/P_k$. Moreover, such orbits are totally ordered, i.e. 
    \[\overline{\Sigma_j}=\bigsqcup_{i=1}^j \Sigma_i \ \ \ , \ \forall j \in [d_{k}^{_G}] \ . \]
    \end{lemma}
\begin{proof}
    The total ordering of the orbit closures in $\sigma_2^\circ(G/P_k)$ is induced by the one of the orbits in the direct product. We already know that the diagonal $G \cdot ([v_{\omega_k}],[v_{\omega_k}])$ and the orbit $G\cdot ([v_{\omega_k}],[w_1\cdot v_{\omega_k}])$ of pairs with Hamming distance $1$ define the same $G$--orbit in $\sigma_2^\circ(G/P_k)$, namely $G/P_k$. The only thing to check is that for any $\{s,s'\}\in \binom{[d_k^{_G}]\setminus \{1\}}{2}$ the orbits $G\cdot ([v_{\omega_k}],[w_s\cdot v_{\omega_k}])$ and $G\cdot ([v_{\omega_k}],[w_{s'}\cdot v_{\omega_k}])$ define two distinct orbits of bisecant points $\Sigma_s$ and $\Sigma_{s'}$. \\
    \indent Exceptional cases and quadrics have already been settled (see Remark \ref{rmk:exceptional}), hence we prove the case of Lagrangian Grassmannians only, adapting arguments from the ones used for Grassmannians and Spinor varieties in \cite{galganostaffolani2022grass, galgano2023spinor}. Consider the Lagrangian Grassmannian 
    \[ \LG_{N}=\Sp_{2N}/P_N \subset \mathbb P(V_{\omega_N}^{C_N})\subsetneq\mathbb P(\bigwedge^N\mathbb C^{2N}) \ . \]
    It is the projective variety of $N$--dimensional linear subspaces of $\mathbb C^{2N}$ which are isotropic (i.e. $W\subset W^\perp$) with respect to a non-degenerate symplectic form $\boldsymbol{\omega}\in \bigwedge^2(\mathbb C^{2N})^\vee$. Fix $(e_1,\ldots, e_N, e_{-N},\ldots,e_{-1})$ a basis of $\mathbb C^{2N}$ such that $\boldsymbol{\omega}(e_i,e_j)=\boldsymbol{\omega}(e_{-i},e_{-j})=0$ and $\boldsymbol{\omega}(e_i,e_{-j})=\delta_{ij}$ for any $i,j\in [n]$, where $\delta_{ij}$ denotes the Kronecker symbol. We denote by $[W]$ the point in $\LG_{N}$ corresponding to the isotropic subspace $W$. We fix the notation $E^+:=\langle e_1,\ldots, e_N\rangle_\mathbb C$ and $E^-:=\langle e_{-N},\ldots, e_{-1}\rangle_\mathbb C$ for the isotropic subspaces corresponding to the highest weight vector $v_{\omega_N}=\bold{e}_{[N]}=[E^+]$ and the lowest weight vector $\ell_{\omega_N}=\bold{e}_{-[N]}=[E^-]$ in $V_{\omega_N}^{C_N}$. \\
    \indent The unique $\Sp_{2N}$--equivariant map $V_{\omega_N}^{C_N}\otimes V_{\omega_1}^{C_N}\rightarrow V_{\omega_{N-1}}^{C_N}$ defined on decomposable elements $\bold{q}_{[N]}=q_1\wedge \ldots \wedge q_N \in V_{\omega_N}^{C_N}$ by
    \[ \bold{q}_{[N]}\otimes v \mapsto v^*(\bold{q}_{[N]}):= \sum_{j \in [N]}(-1)^{j+1}\boldsymbol{\omega}(v,q_j)\bold{q}_{[N]\setminus \{j\}} \]   
    and extended by linearity, induces for any $q\in \bigwedge^N\mathbb C^{2N}$ the linear map
    \begin{equation}\label{eq:kernels}
        \begin{matrix}
            \psi_{q}: & \mathbb C^{2N} & \longrightarrow & \bigwedge^{N-1}\mathbb C^{2N}\\
            & v & \mapsto & v^*(q)
            \end{matrix}
        \end{equation}
    For any $q \in V_{\omega_N}^{C_N}$ we denote the kernel of $\psi_q$ by $H_q:=\Ker(\psi_{q})$. Note that it holds $H_{[W]}=W$ for any $[W]\in \LG_{N}$. \\
    \indent From \cite[Lemma 4.2]{buch2013} and the proof of \cite[Prop.\ 4.5]{buch2013} we know that $\Sp_{2N}$--orbits in $\LG_{N}\times \LG_{N}$ are parameterised by the Hamming distance between points in a pair, which is the codimension of the intersection of the corresponding subspaces:
    \[ d([V],[W])=N-\dim(V\cap W) \ \ \ \ , \ \forall [V],[W]\in \LG_{N}. \]
    Thus a set of representatives for the $\Sp_{2N}$--orbits in $\LG_{N}\times \LG_{N}$ is given by
    \[ \left(\bold{e}_{[N]}, \bold{e}_{[N]\setminus[s]}\wedge \bold{e}_{-[s]}\right) \ \ \ \ , \ \forall s\in [N] \ . \]
    For any $s\in [N]\setminus \{1\}$ the bisecant point 
    \[ \bold{x}_s:=\left[\bold{e}_{[N]}+ \bold{e}_{[N]\setminus[s]}\wedge \bold{e}_{-[s]}\right]= \left[\bold{e}_{[N]\setminus [s]}\wedge (\bold{e}_{-[s]} + (-1)^s\bold{e}_{[s]})\right]\] in $\sigma_2^\circ(\LG_{N})$ is a representative for the orbit $\Sigma_s$ and it defines the $(N-s)$--dimensional subspace $H_{\bold{x}_s}= \langle e_{s+1},\ldots, e_N\rangle_\mathbb C$. Since for any $g \in \Sp_{2N}$ and $q \in V_{\omega_N}^{C_N}$ it holds $H_{g\cdot q}=g H_q g^{-1}$, the action of $\Sp_{2N}$ preserves the dimensions of such subspaces, hence for any $\{s,s'\}\in \binom{[N]\setminus \{1\}}{2}$ the points $\bold{x}_s$ and $\bold{x}_{s'}$ cannot be conjugated by the $\Sp_{2N}$--action. We deduce that the $N$ orbits $\Sigma_j$ for $j\in [N]$ are all distinct.
    \end{proof}

\begin{es}[Grassmannian of planes in $\mathbb C^4$]
    The $A_3$--type cominuscule variety $\SL_4/P_2\subset \mathbb P(V_{\omega_2}^{A_3})$ is the Grassmannian $\Gr(2,4)\subset \mathbb P(\bigwedge^2\mathbb C^4)$ in its (minimal) Pl\"{u}cker embedding. Given the simple roots $\Delta=\{\alpha_i:=e_i-e_{i+1} \ | \ i\in [3]\}$, the root system associated to $P_2$ is $\Phi^+_2=\{ \alpha_2, \alpha_1 + \alpha_2, \alpha_2 + \alpha_3 , \alpha_1 + \alpha_2 + \alpha_3\}$. The simple reflection $s_{\alpha_i}$ in the Weyl group $\cal W_{\SL_4}$ corresponds to the transposition $(i \ i+1)\in \mathfrak{S}(e_1,e_2,e_3,e_4)$. The Weyl groups of $\SL_4$ and $P_2$ are respectively the symmetric group $\cal W_{\SL_4}=\langle s_{\alpha_1},s_{\alpha_2},s_{\alpha_3}\rangle \simeq \mathfrak{S}_4$ and the Klein group $\cal W_{P_2}=\langle s_{\alpha_1},s_{\alpha_3}\rangle \simeq (\mathbb Z/2\mathbb Z)^{\times 2}$. \\
    After identifying every Schubert cell $Bw_\sigma P$ with the Weyl representative $w_\sigma \in \cal W_{\SL_4}$, the Bruhat decomposition of $\Gr(2,4)$ into Schubert cells is controlled by the $6$ cosets in $\cal W_{\SL_4}/\cal W_{P_2}$, whose representatives lead to the Hasse diagram in \autoref{figure:Bruhat G(2,4).a}. On the other hand, $P_2$--orbits in $\Gr(2,4)$ are in bijection with the double cosets in $\cal W_{P_2}\!\!\setminus \!\cal W_{\SL_4}/\cal W_{P_2}$. In particular, the left-action of $\cal W_{P_2}$ conjugates the vertices of the rhombus in the diagram in \autoref{figure:Bruhat G(2,4).a}, leading to the Hasse diagram in \autoref{figure:Bruhat G(2,4).b}. 
    
    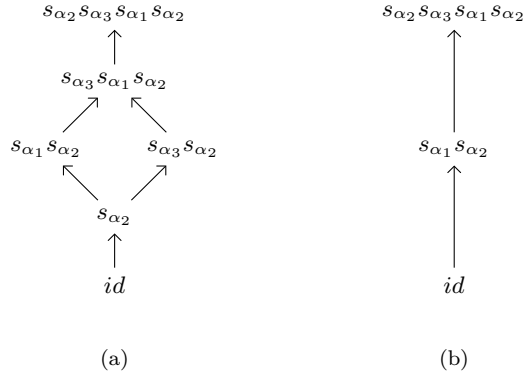
\begin{figure}[h]
	\begin{center}
        \begin{subfigure}[b]{0.3\linewidth}
        \begin{center}
		{\small \begin{tikzpicture}[scale=1.6]
			
			\node(id) at (0,0){{$id$}};
			\node(s2) at (0,0.5){{$s_{\alpha_2}$}};
			\node(s12) at (-0.5,1){{$s_{\alpha_1}s_{\alpha_2}$}};
			\node(s32) at (0.5,1){{$s_{\alpha_3}s_{\alpha_2}$}};
			\node(s312) at (0,1.5){$s_{\alpha_3}s_{\alpha_1}s_{\alpha_2}$};
			\node(s2312) at (0,2){$s_{\alpha_2}s_{\alpha_3}s_{\alpha_1}s_{\alpha_2}$};
			
			\path[font=\scriptsize,>= angle 90]
			(id) edge [->] node [left] {} (s2)
			(s2) edge [->] node [left] {} (s32)
			(s2) edge [->] node [left] {} (s12)
			(s32) edge [->] node [left] {} (s312)
			(s12) edge [->] node [left] {} (s312)
			(s312) edge [->] node [left] {} (s2312);
			\end{tikzpicture}}  
            \end{center}
            \caption{}
	        \label{figure:Bruhat G(2,4).a}
        \end{subfigure}
        \begin{subfigure}[b]{0.3\linewidth}
            \begin{center}
            {\small \begin{tikzpicture}[scale=1.6]
			
			\node(id) at (0,0){{$id$}};
			\node(s2) at (0,1){{$s_{\alpha_1}s_{\alpha_2}$}};
			\node(s2312) at (0,2){$s_{\alpha_2}s_{\alpha_3}s_{\alpha_1}s_{\alpha_2}$};
			
			\path[font=\scriptsize,>= angle 90]
			(id) edge [->] node [left] {} (s2)
			(s2) edge [->] node [left] {} (s2312);
			\end{tikzpicture}} 
            \end{center}
            \caption{}
	        \label{figure:Bruhat G(2,4).b}
        \end{subfigure}
	\end{center}
	\caption{Hasse diagrams of Schubert cells (\ref{figure:Bruhat G(2,4).a}) and of $P_2$--orbits (\ref{figure:Bruhat G(2,4).b}) in $\Gr(2,4)$.}
	\label{figure:Bruhat G(2,4)}
\end{figure}

    \noindent Let's re-analyze the $P_2$--orbits in $\Gr(2,4)$ in terms of Proposition \ref{prop:secant orbit cominuscule}. A maximal orthogonal sequence of roots in $\Phi^+_2$ is $(\alpha_1+\alpha_2, \alpha_2+\alpha_3)$. The first orbit is given by $P_2$ itself, represented by $id$. The reflection $s_{\alpha_1+\alpha_2}$ corresponds to the transposition $(1 \ 3)=(1 \ 2)(2 \ 3) (1 \ 2)$, hence $s_{\alpha_1+\alpha_2}=s_{\alpha_1}s_{\alpha_2}s_{\alpha_1}$. This gives the second orbit 
    \[P_2s_{\alpha_1+\alpha_2}P_2=P_2s_{\alpha_1}s_{\alpha_2}P_2 \ . \]
    Finally, accordingly to Proposition \ref{prop:secant orbit cominuscule}, the third orbit has representative $s_{\alpha_1+\alpha_2}s_{\alpha_2+\alpha_3}$, corresponding to the permutation $(1 \ 3)(2 \ 4)$. Similarly to the previous case, one gets $s_{\alpha_1+\alpha_2}s_{\alpha_2+\alpha_3}=(s_{\alpha_3}s_{\alpha_2}s_{\alpha_3})(s_{\alpha_1}s_{\alpha_2}s_{\alpha_1})$, leading to the third orbit
    \[ P_2 s_{\alpha_1+\alpha_2}s_{\alpha_2+\alpha_3} P_2 = P_2 s_{\alpha_2}s_{\alpha_3}s_{\alpha_1}s_{\alpha_2} P_2 \ .\]
    \noindent We conclude this example by exhibiting the orbits in the dense subset $\sigma_2^\circ(\Gr(2,4))\subseteq \mathbb P(\bigwedge^2\mathbb C^4)$ (in this case the last inclusion is actually an equality). Recall that the highest weight vector in $\bigwedge^2\mathbb C^4$ is $v_{\omega_2}=e_1\wedge e_2$. From Lemma \ref{lemma:secant orbits} we know that the two lower orbits in \autoref{figure:Bruhat G(2,4).b} collapse to a unique orbit in $\sigma_2^\circ(\Gr(2,4))$, and the only $\SL_4$--orbits are $\Gr(2,4)$ and $\SL_4\cdot [(e_1\wedge e_2) + (s_{\alpha_2}s_{\alpha_3}s_{\alpha_1}s_{\alpha_2})\cdot (e_1\wedge e_2)]$. Since $s_{\alpha_2}s_{\alpha_3}s_{\alpha_1}s_{\alpha_2}$ corresponds to the permutation $(13)(24)$, one gets $(s_{\alpha_2}s_{\alpha_3}s_{\alpha_1}s_{\alpha_2})\cdot (e_1\wedge e_2)= e_3\wedge e_4$ and recover the orbit partition
    \[ \sigma_2^\circ(\Gr(2,4))= \Gr(2,4) \sqcup \SL_4\cdot [e_1\wedge e_2 + e_3\wedge e_4] \ , \]
    corresponding to rank--$2$ and rank--$4$ skew-symmetric matrices of size $4\times 4$ (cf. \cite[Sec.\ 2]{galganostaffolani2022grass}).
    \end{es}

\begin{es}[Grassmannians]
    Consider the Grassmannian $\Gr(k,N)=A_N/P_k$ in its minimal Pl\"{u}cker embedding $\mathbb P(V_{\omega_k}^{A_N})=\mathbb P(\bigwedge^k\mathbb C^N)$. Set $\Delta:=\{\alpha_i:=e_i-e_{i+1} \ | \ i\in [N-1]\}$ for the simple roots, $\cal W_{\SL_N}=\langle s_{\alpha_i} \ | \ i\in [N-1]\rangle \simeq \mathfrak{S}_N$ and $\cal W_{P_k}=\langle s_{\alpha_i} \ | \ i\in [N-1]\setminus \{k\}\rangle \simeq \mathfrak{S}_k \times \mathfrak{S}_{N-k}$ for the Weyl groups of $\SL_N$ and $P_k$ respectively. The Schubert cells in $\Gr(k,N)$ are in bijection with the $\binom{N}{k}$ cosets $\cal W_{\SL_N}/\cal W_{P_k}$. A maximal orthogonal sequence of (long) roots in the root system $\Phi_k^+=\{ e_i-e_{j+1} \ = \ \sum_{\ell \in [j]\setminus [i-1]}\alpha_\ell \ | \ i\in [k] \ , \ j \in [N-1]\setminus [k-1]\}$ associated to $P_k$ is 
    
    \begin{align*}
        \beta_1 & = \alpha_1 + \ldots + \alpha_k = e_1-e_{k+1}\\
        \beta_2 & = \alpha_2 + \ldots + \alpha_{k+1} = e_2-e_{k+2}\\
        \vdots \\
        \beta_{k-1} & = \alpha_{k-1} + \ldots + \alpha_{2k-2} = e_{k-1} - e_{2k-1} \\
        \beta_k & = \alpha_k + \ldots + \alpha_{2k-1} = e_k-e_{2k}
        \end{align*}

\noindent defining the reflections $s_{\beta_j}=(j \ k+j)$ and the Weyl representatives $w_j:=(1 \ k+1)(2 \ k+2) \cdots (j \ k+j)$ for $j \in [k]$. By Lemma \ref{lemma:secant orbits} the subset $\sigma_2^\circ(\SL_N/P_k)$ splits in the $\SL_N$--orbits
\[ \SL_N\cdot [v_{\omega_k}+w_j\cdot v_{\omega_k}] = \SL_N \cdot \left[\bold{e}_{[k]} + \bold{e}_{[k+j]\setminus [k]}\wedge \bold{e}_{[k]\setminus [j]}\right] \ \ \ \ \ , \ \forall j \in [k]\]
where $\bold{e}_{I}:=e_{i_1}\wedge\ldots \wedge e_{i_d}$ for any $I \subset \binom{[N]}{d}$. The above orbit partition is precisely the one already obtained in \cite[Proposition 3.1.1]{galganostaffolani2022grass}.
    \end{es}

\subsection{Orbits in $\tau(G/P_k)$} 

We are left with analyzing the tangent orbits in $\sigma_2(G/P_k)$, or equivalenlty the $P_k$--orbits in $\mathfrak{g}/\mathfrak{p}_k\simeq \mathfrak{p}_k^u$. As the variety $G/P_k$ is cominuscule, the tangent space $\mathfrak{g}/\mathfrak{p}_k=\mathfrak{g}_{-1}$ is an irreducible $P_k$--module. We recall some arguments due to L. Hille and G. R\"{o}hrle \cite{hillerohrle} implying that for cominuscule varieties such orbits are finitely many. \\
\indent Given a parabolic subgroup $P_I$ defined by the subset $I\subset \Delta$ of simple roots, the {\em nilpotency class} of the radical unipotent $P^u$ (or equivalently of the nilpotent algebra $\mathfrak{p}^u$) is 
\begin{equation}\label{eq:nilpotency class}
\ell(P^u) := \sum_{\alpha \in \Delta\setminus \Phi(I)^0}m_\alpha(\rho) \ , 
\end{equation}
where $\rho$ is the longest root in $\Phi$, $m_\alpha(\rho)$ the coefficient of $\alpha$ in $\rho$, and $\Phi(I)^0$ are the roots of the parabolic algebra $\mathfrak{p}_I$. For $I=\{\alpha_k\}$ and $\omega_k$ cominuscule, one gets $\ell(P_k^u)=m_{\alpha_k}(\rho)=1$. Then \cite[Theorem 1.1]{hillerohrle} implies that, for any cominuscule variety $G/P_k$, the tangent space $\mathfrak{g}/\mathfrak{p}_k$ has finitely many $P_k$--orbits. Next proposition gives the exact numbers of such orbits, which we determine by analysing the tangent spaces $\mathfrak{g}/\mathfrak{p}_k$ case by case.

\begin{prop}\label{prop:tangent orbit cominuscule}
	For $G/P_k$ cominuscule, the tangent space $T_{[v_{\omega_k}]}(G/P_k)\simeq \mathfrak{g}/\mathfrak{p}_k$ splits in $(d_{k}^{_G}+1)$--many orbits for $d_{k}^{_G}$ as in \autoref{table:secants to cominuscule}, and they are totally ordered:
	\[ T_{[v_{\omega_k}]}(G/P_k) = \{[v_{\omega_k}]\} \sqcup \bigsqcup_{j=1}^{d_{k}^{_G}}\cal R_j \ \ \ \ , \ \ \ \ \ \overline{\cal R_i}=\{[v_{\omega_k}]\} \sqcup \bigsqcup_{s=1}^i\cal R_s \ \ , \ \forall i\in [d_{k}^{_G}] \ . \]
\end{prop}
\begin{proof}
	The unipotent radical $P_k^u$ being abelian, it acts trivially on $\mathfrak{g}/\mathfrak{p}_k\simeq \mathfrak{p}_k^u$, hence the parabolic action on the tangent space depends on the Levi component only. \\
    \indent The tangent spaces to Grassmannians, Spinor varieties and Lagrangian Grassmannians are the spaces of matrices $\mathfrak{sl}_N/\mathfrak{p}_k\simeq \mathbb C^k \otimes \mathbb C^{N-k}$, $\mathfrak{so}_{2N}/\mathfrak{p}_N\simeq \bigwedge^2\mathbb C^N$ (skew-symmetric) and $\mathfrak{sp}_{2N}/\mathfrak{p}_N\simeq \Sym^2\mathbb C^N$ (symmetric) with the parabolic actions of $\SL_k\times \SL_{N-k}$ and $\SL_N$ respectively. Therefore the orbit stratification is parametrised by the matrix-rank only, and the maximum ranks are respectively $d_{k}^{_{A_N}}=\min\{k,N-k\}$, $d_{N}^{_{D_N}}=\lfloor \frac{N}{2}\rfloor$ and $d_N^{_{C_N}}=N$. \\
	\indent The case of $E_7/P_7$ is already known among the Legendrian varieties (cf. Remark \ref{rmk:exceptional})): the $d_{7}^{_{E_7}}=3$ orbits in $T_{[v_{\omega_7}]}(E_7/P_7)$ are obtained as intersections with the orbits $E_7/P_7$, $\sigma_+\setminus (E_7/P_7)$ and $\tau(E_7/P_7)\setminus \sigma_+$, which are described by matrices of rank $1,2,3$ too. \\
	\indent Finally, the case of the quadric $\Q_{2N-2}$ is trivial, while the secant variety of lines to the Cayley plane $\OP^2$ is defective and $\sigma_2(\OP^2)=\tau(\OP^2)$: in particular, there are $d_{1}^{_{E_6}}=2$ orbits.
\end{proof}

\indent For each $P_k$--orbit $\cal R_j$ in the tangent space $T_{[v_{\omega_k}]}(G/P_k)$ as above, let $\theta_j\in \cal R_j$ be a representative, corresponding to a rank--$j$ matrix in $\mathfrak{g}/\mathfrak{p}_k$. We denote its $G$--orbit in the tangential variety $\tau(G/P_k)$ by
\[ \Theta_j:= G \cdot \mathcal{R}_j \ = \ G\cdot [v_{\omega_k} + \theta_j] \ \ \ \ \ , \ \forall j\in [d_{k}^{_G}] \ . \]
Clearly, it holds $\Theta_j \cap T_{[v_{\omega_k}]}(G/P_k)=\cal R_j$. Each point $[p]\in G/P_k$ with Hamming distance $1$ from $[v_{\omega_k}]$ (i.e. $[p]\in P_kw_1P_k$ in the notation of Proposition \ref{prop:secant orbit cominuscule}) defines a tangent direction in $T_{[v_{\omega_k}]}(G/P_k)$. Then the point $[v_{\omega_k}+p]$ lies both in the tangent orbit $\Theta_{j(p)}$ for a certain $j(p)\in [d_k^{_G}]$, and in $G/P_k$ since the pair $([v_{\omega_k}],[p])$ has Hamming distance $1$. In particular, it follows $\Theta_{j(p)}=G \cdot [v_{\omega_k}+p] \subset G/P_k$, and from the ordering among the orbits in Proposition \ref{prop:tangent orbit cominuscule} one deduces $j(p)=1$ and 
\[ \Theta_1=G/P_k \ \ \ , \ \ \ \mathcal{R}_1= T_{[v_{\omega_k}]}(G/P_k) \cap (G/P_k) \ . \]

\begin{cor}\label{cor:tangent orbits}
    For $G/P_k$ cominuscule, the tangential variety $\tau(G/P_k)$ splits in the $d_k^{_G}$--many $G$--orbits
    \[ \Theta_j \ := \ G \cdot \mathcal{R}_j \ = \ G \cdot [v_{\omega_k}+\theta_j] \ \ \ \ \ , \ \forall j \in [d_k^{_G}]\]
    where $\theta_j\in T_{[v_{\omega_k}]}(G/P_k)$ is a tangent point corresponding to a rank--$j$ matrix in $\mathfrak{g}/\mathfrak{p}_k$. In particular, $\Theta_1=G/P_k$. Moreover, such orbits are totally ordered, i.e.
    \[ \overline{\Theta_j} = \bigsqcup_{i \in [j]}\Theta_i \ \ \ , \ \forall j \in [d_k^{_G}] \ . \]
    \end{cor}

\begin{es}[Lagrangian Grassmannians]\label{es:tangent_Lagrangian}
    We assume the setting for $\LG_{N}$ as in the proof of Lemma \ref{lemma:secant orbits}. The tangent space at a point $[W]\in \LG_{N}$ is isomorphic to the space $\Sym^2W^\vee$ (indeed, the tangent bundle on $\LG_{N}$ is $\Sym^2\cal U^\vee$ for $\mathcal U$ being the pullback of the universal bundle on the Grassmannian $\Gr(N,2N)$). The isomorphism is given by restriction of the one on Grassmannians $T_{[W]}\Gr(N,2N)=\bigwedge^{N-1}W\wedge \mathbb C^{2N} \simeq W^\vee \otimes (\mathbb C^{2N}/W)$: the additional condition of isotropicity $W=W^\perp$ implies $\mathbb C^{2N}/W\simeq W^\vee$, and the right-hand side restricts to $\Sym^2W^\vee \subset W^\vee \otimes W^\vee$. For instance, in the tangent space $T_{[v_{\omega_N}]}\LG_{N} \subseteq \bigwedge^{N-1}E^+\wedge \mathbb C^{2N}$ at the highest weight vector $[v_{\omega_N}]=[\bold{e}_{[N]}]=[E^+]$, for any $j\in [d_N^{_{\Sp_{2N}}}]=[N]$ the rank--$j$ symmetric matrix $\diag(\bold{Id}_j, \bold{0}_{N-j})\in \Sym^2(E^+)^\vee$ corresponds to the point 
    \[ \theta_j := \sum_{i\in [j]} (-1)^{i+1}\bold{e}_{[N]\setminus \{i\}}\wedge e_{-i} \ \in \ T_{[\bold{e}_{{[N]}}]}\LG_{N} \ . \]
    In particular, the point $[\theta_j]$ is a representatives of the $\Sp_{2N}$--orbit $\Theta_j$ in the tangential variety. Note that for $j=1$ one gets the point $[\theta_1]=[e_{-1}\wedge \bold{e}_{[N]\setminus \{1\}}]$ lying in $\LG_{N}$ itself, and
    \[ \mathcal{R}_1 = T_{\bold{e}_{[N]}}\LG_{N} \cap \LG_{N} \simeq \nu_2(\mathbb P^{N-1}) \ . \]
    \end{es}

\subsection{Poset in $\sigma_2(G/P_k)$}

Now that we have the (totally ordered) orbit posets of the dense subset $\sigma_2^\circ(G/P_k)$ and of the tangential variety $\tau(G/P_k)$, we are left with relating the secant and the tangent orbits each other. A first relation we obtain is $G/P_k = \Sigma_1 = \Theta_1$.\\
\indent Note that any point $q \in \cal R_2$ corresponds to a rank--$2$ matrix, hence it is of the form $q=p_1+p_2$ for $p_1,p_2\in \cal R_1$ (corresponding to) rank--$1$ matrices. In particular, from the equality $\Theta_1=G/P_k$, both points $[v_{\omega_k}+p_1]$ and $[v_{\omega_k}+p_2]$ lie in $G/P_k$, hence the point $[v_{\omega_k}+q]=[(\frac{1}{2}v_{\omega_k}+p_1)+(\frac{1}{2}v_{\omega_k}+p_2)]$ lies on the bisecant line $L([\frac{1}{2}v_{\omega_k}+p_1],[\frac{1}{2}v_{\omega_k}+p_2])$ whose defining pair $([\frac{1}{2}v_{\omega_k}+p_1],[\frac{1}{2}v_{\omega_k}+p_2])$ has Hamming distance $2$ (since the two points are linked by the two lines passing through $[v_{\omega_k}]$). We conclude that $\Sigma_2 = \Theta_2$.\\
\indent The next result gives a complete description of the partially ordered set of $G$--orbits in $\sigma_2(G/P_k)$ for any $G/P_k$ cominuscule variety, proving Conjecture 6.1.3 in \cite{galganothesis}.

\begin{teo}\label{thm:poset cominuscule}
	Let $G/P_k$ be a cominuscule variety and let $d_k^{_G}$ be as in \autoref{table:secants to cominuscule}. The secant variety of lines $\sigma_2(G/P_k)$ splits in the finitely many $G$--orbits $\Sigma_j$ and $\Theta_j$ as defined in Lemma \ref{lemma:secant orbits} and Corollary \ref{cor:tangent orbits}, for $j \in [d_k^{_G}]$. Their poset is described by the graph in \autoref{figure:graph cominuscule}, where arrows denote the inclusions among orbit closures.
    \begin{figure}[H]
	\begin{center}
		{\small \begin{tikzpicture}[scale=1.8]
			
			\node(S) at (0,0.2){{$G/P_k$}};
			\node(t2) at (0,0.7){{$\Theta_2=\Sigma_2$}};
			\node(t3) at (-0.6,1.1){{$\Theta_3$}};
			\node(t) at (-0.6,1.6){{$\vdots$}};
			\node(td) at (-0.6,2.1){$\Theta_{d_{k}^{_G}}$};
			
			\node(s3) at (0.6,1.3){{$\Sigma_3$}};
			\node(s) at (0.6,1.8){{$\vdots$}};
			\node(sd) at (0.6,2.3){{$\Sigma_{d_{k}^{_G}}$}};
			
			\path[font=\scriptsize,>= angle 90]
			(S) edge [->] node [left] {} (t2)
			(t2) edge [->] node [left] {} (t3)
			(t3) edge [->] node [left] {} (s3)
			(t3) edge [->] node [left] {} (t)
			(t) edge [->] node [left] {} (td)
			(t) edge [->] node [left] {} (s)
			(td) edge [->] node [left] {} (sd)
			(s3) edge [->] node [left] {} (s)
			(s) edge [->] node [left] {} (sd);
			\end{tikzpicture}}
	\end{center}
	\caption{Poset graph of $G$--orbits in $\sigma_2(G/P_k)$ for $G/P_k$ cominuscule.}
	\label{figure:graph cominuscule}
\end{figure}
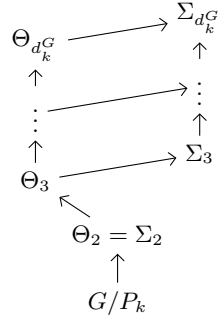
\end{teo}
\begin{proof}
	For the proof in the case of Grassmannians and Spinor varieties we refer to \cite{galganostaffolani2022grass, galgano2023spinor} respectively. Moreover, quadrics, the Cayley plane $\OP^2$ and $E_7/P_7$ have been settled in Remark \ref{rmk:exceptional}. The only missing case is the Lagrangian Grassmannian $\LG_{N}$.\\
    \indent Consider $\LG_{N}$ in its minimal embedding $\mathbb P(V_{\omega_N}^{C_N})\subset \mathbb P(\bigwedge^N\mathbb C^N)$, and look at it inside the Grassmannian $\Gr(N,2N)\subset \mathbb P(\bigwedge^N\mathbb C^{2N})$: more precisely, $\LG_{N}$ is obtained as linear section of $\Gr(N,2N)$ cutted by $\mathbb P(V_{\omega_N}^{C_N})$. Note that the number of orbits in each branch (the tangential and the bisecant ones) for the two varieties coincide, namely $d_{N}^{_{\SL_{2N}}}=d_{N}^{_{\Sp_{2N}}}=N$.
    \indent Thus the orbit graph for $\LG_{N}$ follows straightforwardly from the one for $\Gr(N,2N)$ by observing that the orbits in $\sigma_2(\LG_{N})$ are obtained by cutting the orbits in $\sigma_2(\Gr(N,2N))$ by $\mathbb P(V_{\omega_N}^{C_N})$, that is
    \[ \Sigma_\ell^{{\Gr(N,2N)}} \cap \mathbb P(V_{\omega_N}^{C_N}) = \Sigma_\ell^{{\LG_{N}}} \ \ \ , \ \ \ \Theta_\ell^{{\Gr(N,2N)}}\cap \mathbb P(V_{\omega_N}^{C_N})=\Theta_\ell^{{\LG_{N}}} \ .\]
\end{proof}

Note that the orbits $\Theta_{d_k^{_G}}$ and $\Sigma_{d_k^{_G}}$ are the dense orbits in the tangential variety $\tau(G/P_k)$ and in the secant variety of lines $\sigma_2(G/P_k)$ respectively.

\begin{obs}\label{rmk:pullback O(1)}
    The reason why the above inductive argument (from Grassmannians) applies to Lagrangian Grassmannians but not to Spinor varieties, is that the pullback of the very ample line bundle $\cal O_{\Gr(N,2N)}(1)$ onto $\LG_{N}$ is $\cal O_{\LG_{N}}(1)$, while the pullback onto $\mathbb S_N$ is $\cal O_{\mathbb S_N}(2)$ (equivalently, $\cal O_{\mathbb S_N}(1)$ is the \textquotedblleft square root \textquotedblright \space of $\cal O_{\Gr(N,2N)}(1)$).
    \end{obs}

\begin{obs*}
    The above proof just relies on a by-hand analysis of the Lagrangian case and on the results and arguments from \cite{galganostaffolani2022grass}. However, although conclusive, it doesn't highlight where such common behaviour among all cominuscule varieties comes from. It would be interesting to find a more general representation-theoretic proof of the inclusions among tangent and bisecant orbits (the horizontal arrows in the orbit graph). 
    \end{obs*}

\section{Identifiability, singular locus and Terracini locus in the cominuscule case}\label{sec:identif_sing}

In this section we show that the secant varieties of lines to cominuscule varieties share also the same identifiable, singular and Terracini loci, completing the proof of our Main Theorem. We assume terminology and notation from Sec.\ \ref{subsec:secant}.

\subsection{Identifiability}\label{subsec:ident cominuscule}

We treat each cominuscule variety separately, recalling the already known cases and completing the missing ones. \\
\indent {\em Quadrics and Cayley plane.} From Remark \ref{rmk:defectivity-unidentifiability}, we know that all points in $\sigma_2(\Q_m)\setminus \Q_m$ and $\sigma_2(\OP^2)\setminus \OP^2$ are unidentifiable and their decomposition loci have dimensions $m$ and 8 respectively. Observe that both complements coincides with the second-to-minimal orbit $\Sigma_2$ in the corresponding secant variety. \\
\indent {\em $E_7/P_7$.} The problem of identifiability in $\sigma_2(E_7/P_7)$ (together with the other Legendrian varieties $\LG_3$, $\Gr(3,6)$ and $\mathbb S_6$) is solved in \cite[Proposition 5.11]{LM01}: the points in $\Sigma_2$ (in their notation, $\sigma_+\setminus X$) are unidentifiable, the ones in $\Theta_3$ (in their notation, $\tau(X)\setminus \sigma_+$) are tangential-identifiable, and the ones in $\Sigma_3$ (i.e. $\sigma_2(X)\setminus \tau(X)$) are identifiable. Moreover, the decomposition locus of each point in $\Sigma_2$ has dimension $10$ ($m+2$ for $m=8$ in their notation, cf. second proof to \cite[Proposition\ 5.11]{LM01}). \\
\indent {\em Grassmannians and Spinor varieties.} The identifiable loci in $\sigma_2(\Gr(k,N))$ and $\sigma_2(\mathbb S_N)$ are determined in \cite[Secc.\ 4-5]{galganostaffolani2022grass} and \cite[Secc.\ 4-5]{galgano2023spinor} respectively, accordingly to the second statement in our Main Theorem. Observe that secant varieties of lines to $\Gr(2,N)$, $\mathbb S_4$ and $\mathbb S_5$ either are defective or overfill their ambient space, with defects $4$ for Grassmannians and $6$ for Spinor varieties. Moreover, such defectivity is the cause of the unidentifiability of points in the second-to-minimal $\Sigma_2$ in any $\sigma_2(\Gr(k,N))$ and $\sigma_2(\mathbb S_N)$.\\
\indent {\em Lagrangian Grassmannians.} The Lagrangian Grassmannian $\LG_2\subset \mathbb P(V^{\Sp_4}_{\omega_2})$ is a $3$--dimensional quadric in $\mathbb P(V^{\Sp_4}_{\omega_2})\simeq \mathbb P^4$, thus points in $\sigma_2(\LG_2)\setminus \LG_2$ are unidentifiable with $3$--dimensional decomposition loci. On the other hand, as a Legendrian variety, the case $\LG_3$ behaves identically to $E_7/P_7$. The following result covers all remaining cases $N\geq 4$ and it is a quite straightforward consequence of what happens for Grassmannians.

\begin{prop}\label{prop:identifiability LG}
    For any $N \geq 2$, the unidentifiable locus of the secant variety of lines $\sigma_2(\LG_{N})$ to the Lagrangian Grassmannian $\LG_{N}$ coincides with the second-to-minimal $\Sp_{2N}$--orbit $\Sigma_2$, as defined in Lemma \ref{lemma:secant orbits}. In particular, the decomposition locus of every point in $\Sigma_2$ has dimension $3$.
    \end{prop}
\begin{proof}
    The same argument in the proof of Theorem \ref{thm:poset cominuscule} and in Remark \ref{rmk:pullback O(1)} implies that points in $\sigma_2(\LG_{N})$ which are identifiable (resp. tangentially identifiable) as points in $\sigma_2(\Gr(N,2N))$ are identifiable (resp. tangential-identifiable) in $\sigma_2(\LG_{N})$ too. In particular, the orbits $\Sigma_\ell \subset \sigma_2^\circ(\LG_{N})$ and $\Theta_\ell\subset \tau(\LG_{N})$ for $3\leq \ell \leq N$ are respectivelyidentifiable and tangential-identifiable. On the other hand, we show that points in the second-to-minimal orbit $\Sigma_2$ are unidentifiable as a consequence of the fact that $\LG_2$ is a quadric. We use notation from the proof of Lemma \ref{lemma:secant orbits}. \\
    \indent The representative of the second-to-minimal orbit $\Sigma_2$ is $[\bold{e}_{[N-2]}\wedge ( e_{N-1}\wedge e_{N} + e_{-N}\wedge e_{-(N+1)} )]$. Observe that $[e_{N-1}\wedge e_N + e_{-N}\wedge e_{-(N-1)}]$ is a point in the Lagrangian Grassmannian of planes $\LG_2$, hence it admits a decomposition locus of dimension $3$. It follows that points in $\Sigma_2$ are unidetifiable with $3$--dimensional decomposition loci.
    \end{proof}

Collecting the above results leads to a complete description of the identifiability of points in the secant varieties of lines to cominuscule varieties.

\begin{teo}[Main Theorem, part $2.$]\label{thm:identifiability}
    Let $G/P_k$ be a cominuscule variety in its minimal homogeneous embedding and let $d_k^{_{G}}$ be as in \autoref{table:secants to cominuscule}. For $j \in [d_k^{_G}]$, let $\Sigma_j$ and $\Theta_j$ be the $G$--orbits in the secant variety of lines $\sigma_2(G/P_k)$, as defined in Lemma \ref{lemma:secant orbits} and Corollary \ref{cor:tangent orbits}. Then the unidentifiable locus of $\sigma_2(G/P_k)$ coincides with the second-to-minimal orbit $\Sigma_2$. Moreover, the decomposition loci of points in $\Sigma_2$ have dimensions as in \autoref{table:decomposition loci}.
    \end{teo}

\begin{table}[H]
	\centering 
	{\small \begin{tabular}{ |P{3.5cm}||P{1.5cm}|P{1.2cm}|P{1.5cm}|P{1.2cm}|P{1.2cm}|P{1.2cm}| }
			\hline
			$G/P_k$ & $\Gr(k,N)$ & $\Q_m$ & $\LG_{N}$ & $\mathbb S_N$ & $\OP^2$ & $E_7/P_7$ \\

            \hline
            Unidentifiable locus & $\Sigma_2$ & $\Sigma_2$ & $\Sigma_2$ & $\Sigma_2$ & $\Sigma_2$ & $\Sigma_2$ \\
            
            \hline
			Dim. of decompos. loci & $4$ & $m$ & $3$ & $6$ & $8$ & $10$ \\
			
			\hline
	\end{tabular} }
	\caption{Unidentifiable loci in $\sigma_2(G/P_k)$ and dimensions of decomposition loci of their points.}
	\label{table:decomposition loci}
\end{table}

\subsection{Singular locus}\label{subsec:sing cominuscule}

Recall that in general every projective variety is singular in its secant variety of lines, unless the latter is the ambient space. In \cite[Theorem 4.2, Remark 4.3]{yoonseo} C. Yoon and H. Seo prove a crucial role of the identifiability in the detection of the singular locus of a secant variety of lines.

\begin{teo}[\cite{yoonseo}]\label{thm:yoonseo}
    Let $X\subset \mathbb P(V)$ be a non-degenerate smooth projective variety. Then the secant variety of lines $\sigma_2(X)\subset \mathbb P(V)$ is smooth along its identifiable locus.
    \end{teo}

\begin{obs}
    The converse to the above theorem is false in general. An example is provided by the secant variety of lines to a Veronese variety $\nu_d(\mathbb P^n)$. In \cite[Theorem 3.3(iii)]{kanev1999chordal} Kanev proves that the singular locus of $\sigma_2(\nu_d(\mathbb P^n))$ is the Veronese variety itself. In particular, tangent points to the Veronese variety are smooth. Still, they do not lie in the identifiable locus: for instance, the tangent point $[x_0^{d-1}x_1]$ lies in both tangent spaces at $[x_0^d]$ and $[x_0^{d-2}x_1^2]$.
    \end{obs}

In our setting, combining the above result with our Theorem \ref{thm:identifiability} gives that the singular locus of the secant variety of lines to a cominuscule variety lies in the closure of the second-to-minimal orbit $\overline{\Sigma_2}$. Let us recall what is already known.\\
\hfill\break
\indent {\em Smooth cases.} Let $G/P_k$ be one of the following cominuscule varieties: a quadric $\Q_m$ (this includes $\Gr(2,4)$, $\Gr(2,5)$, $\LG_2$ and $\mathbb S_4$); one of the Legendrian varieties $\Gr(3,6)$, $\LG_3$, $\mathbb S_6$ and $E_7/P_7$; a Spinor variety $\mathbb S_N$ for $N\leq 5$. Then $\sigma_2(G/P_k)$ coincides with its ambient space, hence is smooth.\\
\indent {\em Cases $d_k^{_G}=2$.} Let $G/P_k$ be a cominuscule variety such that $\overline{\Sigma_2}=\sigma_2(G/P_k)\subsetneq \mathbb P(V_{\omega_k}^G)$ (in particular, $d_k^{_G}=2$): this includes the Grassmannians of planes $\Gr(2,N)$ for $N\geq 6$ and the Cayley plane $\OP^2$. Then $\Sing(\sigma_2(G/P_k))=G/P_k$.\\
\indent {\em Higher Grassmannians.} For every $N \geq 7$ and $3\leq k\leq \frac{N}{2}$, the singular locus of $\sigma_2(\Gr(k,7))$ is the closure of the second-to-minimal orbit $\overline{\Sigma_2}$ \cite[Theorem 7.0.9]{galganostaffolani2022grass}. Its dimension has been computed in \cite[Proposition 3.3.6]{galganostaffolani2022grass} (cf. \autoref{table:dim sing loci}). \\
\indent {\em Higher Spinor varieties.} \cite[Corollary 8.3]{galgano2023spinor} proves that, for every $N \geq 7$, the closure of the second-to-minimal orbit is singular while the secant orbits $\Sigma_\ell$ for $\ell \geq 3$ are smooth, that is $\overline{\Sigma_2}\subseteq \Sing(\sigma_2(\mathbb S_N))\subseteq \tau(\mathbb S_N)$. Combining this with Theorem \ref{thm:yoonseo}, we get that for $N\geq 7$ the singular locus of $\sigma_2(\mathbb S_N)$ is the closure $\overline{\Sigma_2}$. Its dimension has been computed in \cite[Proposition 6.2]{galgano2023spinor} (cf. \autoref{table:dim sing loci}).\\
\hfill\break
\indent It only remains to analyse the case of Lagrangian Grassmannians higher than $\LG_3$. For further pourposes, first we compute the dimension of the second-to-minimal orbit $\Sigma_2\subset \sigma_2(\LG_{N})$.

\begin{lemma}\label{lemma:dim Sigma2 LG}
    For any $N\geq 4$, the second-to-minimal orbit $\Sigma_2\subset \sigma_2(\LG_{N})$ has dimension
    \[ \dim \Sigma_2 = \frac{N(N+1)}{2}+2N-3 \ . \]
    \end{lemma}
\begin{proof}
    Let $\IG(N-2,2N)\subset \mathbb P(\bigwedge^{N-2}\mathbb C^{2N})$ be the Grassmannian of $(N-2)$--dimensional subspaces in $\mathbb C^{2N}$ which are isotropic with respect to a non-degenerate symplectic form $\boldsymbol{\omega}\in \bigwedge^2(\mathbb C^{2N})^\vee$ (the same defining $\LG_{N}$). Consider the fibration 
    \[ \begin{matrix} \rho: & \Sigma_2 & \longrightarrow & \IG(N-2,2N) \\ 
        & [\bold{x}+\bold{y}] & \mapsto & H_{\bold{x}+\bold{y}}=H_\bold{x}\cap H_\bold{y} \end{matrix} \ . \]
    where $H_{\bold{x}}$ (and similars) are defined as in the proof of Lemma \ref{lemma:secant orbits}. For any $[W]\in \IG(N-2,2N)$, let $\bold{w}=w_1\wedge \ldots \wedge w_{N-2}$ be such that $W=H_{\bold{w}}$. Observe that, given $W^\perp$ the orthogonal of $W$ with respect to $\boldsymbol{\omega}$, the orthogonal quotient $W^\perp/W$ has dimension $2N-2(N-2)=4$. \\
    \indent The fiber of $\rho$ at $[W]$ is
    \begin{align*}
        \rho^{-1}([W]) & = \big\{ [\bold{x}+\bold{y}] \in \Sigma_2 \ | \ H_\bold{x}\cap H_\bold{y} = W \big\}\\
        & = \big\{ [\bold{w}\wedge (\bold{a} + \bold{b})] \ | \ [\bold{a}],[\bold{b}] \in \LG(2, W^\perp/W) \ , \ H_{\bold{a}}\cap H_{\bold{b}}=\{0\} \big\} \\
        & \simeq \big\{ [\bold{a}+ \bold{b}] \in \sigma_2(\LG(2,W^\perp/W)) \ | \ d([a],[b])=2 \big\} \ .
        \end{align*}
    In particular, the closure of the fiber is isomorphic to the secant variety of lines to a Lagrangian Grassmannian of planes $\sigma_2(\LG_2)=\mathbb P^4$.\\
    \indent On the other hand, the dimension of the isotropic Grassmannian $\IG(N-2,2N)$ is the well-known value (see Remark \ref{rmk:dim IG} for a proof)
    \[\dim \IG(N-2,2N) = (N-2) \frac{4N-3(N-2)+1}{2} = \frac{(N-2)(N+7)}{2} \ . \]
    The fiber dimension theorem leads to the thesis: in the statement, we collect terms in the expression in order to highlight that $\LG_{N}$ has codimension $2N-3$ in $\overline{\Sigma_2}$.
    \end{proof}

\begin{prop}\label{prop:singular LG}
    For any $N\geq 4$, the secant variety of lines to the Lagrangian Grassmannian $\LG_{N}$ is singular along the second-to-minimal orbit $\Sigma_2$.
    \end{prop}
\begin{proof}
    The argument is the same as in \cite[Lemma 8.1]{galgano2023spinor}. Assume that the orbit $\Sigma_2$ is smooth. Then the orbit poset in \autoref{figure:graph cominuscule} implies that the open subset $\sigma_2(\LG_N)\setminus \LG_N$ is smooth. The projection $\pi_\sigma$ from the abstract secant variety 
    \[ Ab\sigma_2(\LG_N):=\overline{\{([x],[y],[q])\in S^2(\LG_N) \times \mathbb P(V_{\omega_N}^{\Sp_{2N}}) \ | \ [q]\in \langle [x],[y]\rangle \}} \]
    onto the second factor $\mathbb P(V_{\omega_N}^{\Sp_{2N}})$ has image the secant variety of lines, and restricts to the map 
    \[ \zeta: Ab\sigma_2(\LG_N)\setminus \pi_\sigma^{-1}(\LG_N) \rightarrow \sigma_2(\LG_N)\setminus \LG_N \ . \]
    Since the unidentifiable locus of $\sigma_2(\LG_N)$ is $\overline{\Sigma_2}$ (see Theorem \ref{thm:identifiability}), the map $\zeta$ is an isomorphism on $Ab\sigma_2(\LG_N)\setminus \zeta^{-1}(\Sigma_2)$ (it is a bijection between smooth open subsets), and its differential drops rank precisely along $\zeta^{-1}(\Sigma_2)$ (and it drops of the dimension of the decomposition locus of points in $\Sigma_2$, that is $3$ - cf. \autoref{table:decomposition loci}). In particular, $\zeta$ is a morphism of smooth varieties of the same dimension (which for $N\geq 4$ is $N(N+1)+1$) and, as such, the locus of points where the differential drops rank has to be a divisor. We show that $\zeta^{-1}(\Sigma_2)$ has not codimension 1, leading to a contradiction.\\
    \indent From the fiber dimension theorem, $\zeta^{-1}(\Sigma_2)$ has dimension $\dim \Sigma_2 + 3$. In order to be a divisor, it must hold 
    \[ \dim \Sigma_2 + 3 = \dim Ab\sigma_2(\LG_N) - 1 = N(N+1) \ \iff \ \dim \Sigma_2 = N(N+1)-3 \ . \]
    But from Lemma \ref{lemma:dim Sigma2 LG} we know that $\dim \Sigma_2 = \frac{N(N+1)}{2}+2N-3$, leading to the contradiction $4N=N(N+1)$.
    \end{proof}

The above proposition combined with Theorem \ref{thm:yoonseo} implies that for $N\geq 4$ the singular locus of the secant variety of lines to $\LG_{N}$ coincides with the closure $\overline{\Sigma_2}$. We collect all (non-trivial) descriptions obtained for cominuscule varieties in the following result, which proves and extends Conjecture 5.6.3 in \cite{galganothesis}.

\begin{teo}[Main Theorem, part $3.$]\label{thm:singular locus}
    Let $G/P_k\subset \mathbb P(V_{\omega_k}^G)$ be a cominuscule variety in its minimal embedding, such that $\sigma_2(G/P_k)\subsetneq \mathbb P(V_{\omega_k}^G)$. Then either $\sigma_2(G/P_k)$ has two $G$--orbits and $\Sing(\sigma_2(G/P_k))=G/P_k$, or
    \[ \Sing(\sigma_2(G/P_k))=\overline{\Sigma_2}=G/P_k \sqcup \Sigma_2 \ ,\]
    where $\Sigma_2$ is the $G$--orbit of points in $\sigma_2(G/P_k)$ lying on both a bisecant line and a tangent line. Moreover, the singular locus has dimension as in \autoref{table:dim sing loci}.
    \end{teo}

    \begin{table}[H]
	\centering 
	{\small \begin{tabular}{ |P{7.5cm}|P{2.5cm}|P{3cm}|  }
			\hline
			$G/P_k$ & $\Sing(\sigma_2(G/P_k))$ & $\dim \Sing(\sigma_2(G/P_k))$ \\
			\hline
            \hline
			$\Q_{m}$ , $\Gr(2,5)$ , $\Gr(3,6)$ , $\LG_3$ , $\mathbb S_5$ , $\mathbb S_6$ , $E_7/P_7$ & $\emptyset$ & $0$\\
			
            \hline
            $\Gr(2,N)$ ($N\geq 6$) & $\Gr(2,N)$ & $2(N-2)$ \\

			\hline
			$\Gr(k,N)$ ($N \geq 7$ , $3\leq k \leq \frac{N}{2}$) & $\overline{\Sigma_2}$ & $k(N-k) +2N-7$\\

			\hline
			$\LG_{N}$ ($N\geq 4$) & $\overline{\Sigma_2}$ & $\frac{N(N+1)}{2}+2N-3$ \\
			
            \hline
			$\mathbb S_N$ ($N\geq 7$) & $\overline{\Sigma_2}$ & $\frac{N(N-1)}{2}+4N-15$\\
			
			\hline
			$\OP^2$ & $\OP^2$ & $16$ \\
			\hline
	\end{tabular} }
	\caption{Dimensions of the singular loci of the secant varieties of lines to cominuscule varieties.}
	\label{table:dim sing loci}
\end{table}

\subsection{Strong-Terracini locus}\label{subsec:terracini locus}

We conclude our study of the geometry of the secant variety of lines to a cominuscule variety by determining the $2$-nd Terracini locus of $G/P_k$, actually a refinement of it which we introduce in the following. Let $\widetilde{\pi_H}$ and $\widetilde{\pi_\sigma}$ be the projections from the (smooth) abstract secant variety $\widetilde{Ab\sigma_2}(G/P_k)$ onto $\Hilb_2(G/P_k)$ and $\sigma_2(G/P_k)$ (cf. Sec.\ \ref{subsec:secant}). 

\begin{Def}\label{def:strong-terracini}
The $2$-nd {\em strong-Terracini locus} of a projective variety $X\subset \mathbb P^M$ is the locus of reduced subschemes in $\Hilb_2(X)$ for which the differential of the projection $\widetilde{\pi_\sigma}: \widetilde{Ab\sigma_2}(X) \rightarrow \sigma_2(X)$ drops {\em generic} rank:
\[ \widetilde{\mathbb{T}_2^{str}}(X) := \overline{ \left\{ \{x,y\}\in \Hilb_2^{red}(X) \ | \ \dim\langle T_xX, T_yX\rangle \lneq \dim \sigma_2(X) \right\} } \ . \]
\end{Def}

\begin{obs}\label{rmk:different terracini}
    The two definitions of Terracini locus agree when $\sigma_2(X)$ is generically identifiable. Otherwise they differ significantly because of Terracini's Lemma: if the generic fiber of $\pi_\sigma$ has positive dimension, then $\mathbb{T}_2(X) = S^2X$ (or equivalently, $\widetilde{\mathbb{T}_2}(X) = \Hilb_2(X)$) -- cf. Remark \ref{rmk:classical Terracini locus and Terracini lemma} --, while $\widetilde{\mathbb{T}_2^{str}}(X)$ is never the whole $\Hilb_2(X)$. For instance, for a quadric $\Q_m$ it holds $\widetilde{\mathbb{T}_2^{str}}(\Q_m)=\emptyset$ while $\widetilde{\mathbb{T}_2}(\Q_m)=\Hilb_2(\Q_m)$.
\end{obs}

\indent In light of Remark \ref{rmk:different terracini}, for $G/P_k=\Q_m, \ \Gr(2,N), \ \mathbb S_5, \ \OP^2$ the classical Terracini locus is $\widetilde{\mathbb{T}_2}(G/P_k)=\Hilb_2(G/P_k)$. We now compute their strong-Terracini locus $\widetilde{\mathbb{T}_2^{str}}(G/P_k)$. 

\begin{prop}\label{prop:terracini locus defective}
    Let $G/P_k\subset \mathbb P(V_{\omega_k}^G)$ be a cominuscule variety minimally embedded, such that $d_{k}^{_G}=2$. Then the $2$-nd strong-Terracini locus $\widetilde{\mathbb{T}_2^{str}}(G/P_k)$ is empty if $G/P_k=\Q_m$, and it is $\widetilde{\pi_H}\circ \widetilde{\pi_\sigma}^{-1}(G/P_k)$ otherwise. 
    \end{prop}
\begin{proof}
    From Terracini's Lemma, $\widetilde{\mathbb{T}_2^{str}}(G/P_k)$ can be either empty or $\widetilde{\pi_H}\circ \widetilde{\pi_\sigma}^{-1}(G/P_k)$.\\
    Fix $x\in G/P_k$. Since $G/P_k$ is intersection of quadrics, for any reduced subscheme $Z=\{a,b\}\in \Hilb_2^{red}(G/P_k)$ such that $x \in \langle a,b \rangle$, it holds $\langle a, b \rangle \subset G/P_k$ (i.e. $a,b\in G/P_k$ have Hamming distance $1$). In particular, any two points on the line $\langle a,b\rangle$ define another reduced subscheme $Z'$ which still is in the fiber $\widetilde{\pi_\sigma}^{-1}(x)$. Therefore all points in the fiber at $x$ can be obtained as follows: pick $a$ to be a point in the intersection $T_x(G/P_k)\cap (G/P_k)$ (these are the lines through $x$ lying in $G/P_k$), and then pick $b$ to be any point in the punctured line $\langle x,a\rangle \setminus \{a\}$. It follows
    \[ \widetilde{\pi_\sigma}^{-1}(x) \simeq \bigg( T_x(G/P_k)\cap G/P_k \bigg) \times \mathbb P^1\setminus \{pt\} \ , \]
    hence $\dim \widetilde{\pi_\sigma}^{-1}(x) = \dim \left[T_x(G/P_k)\cap (G/P_k)\right] +1$. The intersection $Y_k:= T_x(G/P_k)\cap (G/P_k)$ is the homogeneous variety whose Dynkin diagram is obtained by the one of $G/P_k$ by removing the $k$--th marked node and marking its adjacent nodes (cf. \cite[Prop.\ 2.5]{LM03}). In particular, for $G/P_k=\Q_m, \Gr(2,N), \mathbb S_5, \OP^2$, the corresponding intersections are respectively 
    \[ Y_k \ = \ \Q_{m-2} \ , \ \mathbb P^1\times \mathbb P^{N-3} \ , \  \Gr(2,5) \ , \ \mathbb S_5 \ . \]
    \indent Now, given $x \in G/P_k$ and $Z=\{a,b\} \in \Hilb_2^{red}(G/P_k)$ such that $x \in \langle Z \rangle$, the kernel of the differential of $\widetilde{\pi_\sigma}$ at $(Z,x)$ is the tangent space of $\widetilde{\pi_\sigma}^{-1}(x)$ at $(Z,x)$, whose dimension is lower bounded by the dimension of the fiber (they coincide if the fiber is reduced!). Then the thesis follows by comparing case by case the affine dimension of the secant variety $\dim \sigma_2(G/P_k) + 1$ with the affine dimension of the span of the tangent spaces
    \begin{align*}
        \dim \langle T_a(G/P_k), T_b(G/P_k)\rangle & \leq 2\left( \dim (G/P_k) +1\right) - \left( \dim \pi_\sigma^{-1}(x)+1 \right) \\ 
        & = 2\dim(G/P_k) - \dim(Y_k) \ . 
        \end{align*}
    \end{proof}

\begin{teo}[Main Theorem, part $4.$]\label{thm:terracini locus cominuscule}
	Let $G/P_k\subset \mathbb P(V_{\omega_k}^G)$ be a cominuscule variety minimally embedded, such that $d_k^{_G}\geq 3$. Then the $2$-nd strong-Terracini locus $\widetilde{\mathbb{T}_2^{str}}(G/P_k)$ corresponds to the distance--$2$ orbit closure $\overline{\Sigma_{2}}$. More precisely,
	\[ \widetilde{\mathbb{T}_2^{str}}(G/P_k) = \left( \widetilde{\pi_H} \circ \widetilde{\pi_\sigma}^{-1}\right) \left( \overline{\Sigma_{2}} \right) \ . \]
\end{teo}
\begin{proof}
    From Theorem \ref{thm:identifiability} we know that points in $\sigma_2^\circ(G/P_k)$ are either identifiable, or unidentifiable with positive dimensional fiber along $\widetilde{\pi_\sigma}$ (the fiber corresponds to the decomposition locus). Then the thesis follows from Remark \ref{rmk:terracini is non-identifiable with positive fiber}. Indeed, the decomposition loci of unidentifiable points define points in the (strong-)Terracini locus, i.e. $( \widetilde{\pi_H} \circ \widetilde{\pi_\sigma}^{-1} ) ( \overline{\Sigma_{2}} ) \subset \widetilde{\mathbb{T}_2^{str}}(G/P_k)$. On the other hand, the (open) set of identifiable points is smooth (from Theorem \ref{thm:identifiability} and Theorem \ref{thm:singular locus}), hence by Zariski's Main Theorem their fibers do not lie in $\widetilde{\mathbb{T}_2^{str}}(G/P_k)$.
\end{proof}

\begin{table}[H]
	\centering 
	{\small \begin{tabular}{ |P{5cm}|P{3cm}|  }
			\hline
			$G/P_k$ & $\widetilde{\mathbb{T}_2^{str}}(G/P_k)$ \\
			\hline
            \hline
			$\Q_{m}$ & $\emptyset$ \\

            \hline
			$\Gr(2,N)$ ($N\geq 5$) , $\mathbb S_5$ , $\OP^2$ & $\widetilde{\pi_H}\circ \widetilde{\pi_\sigma}^{-1}(G/P_k)$ \\

			\hline
			$\Gr(k,N)$ ($N \geq 6$ , $3\leq k \leq \frac{N}{2}$) & $\widetilde{\pi_H}\circ \widetilde{\pi_\sigma}^{-1}(\overline{\Sigma_2})$ \\

			\hline
			$\LG_{N}$ ($N\geq 3$) & $\widetilde{\pi_H}\circ \widetilde{\pi_\sigma}^{-1}(\overline{\Sigma_2})$  \\
			
            \hline
			$\mathbb S_N$ ($N\geq 6$) & $\widetilde{\pi_H}\circ \widetilde{\pi_\sigma}^{-1}(\overline{\Sigma_2})$ \\

            \hline
			$E_7/P_7$ & $\widetilde{\pi_H}\circ \widetilde{\pi_\sigma}^{-1}(\overline{\Sigma_2})$ \\
			\hline
	\end{tabular} }
	\caption{Terracini loci of the secant varieties of lines to cominuscule varieties.}
	\label{table:terracini loci}
\end{table}

\section{A non-cominuscule example: $\IG(k,2N)$}\label{sec:isotropicgrassmannian}


This section is devoted to show that the graph in \autoref{figure:graph cominuscule} does not hold for any generalised Grassmannians. As a counterexample, for $2\leq k\lneq N$ we consider the isotropic Grassmannian $\IG(k,2N)=C_N/P_k$ in its minimal embedding $\mathbb P(V_{\omega_k}^{C_N})=\mathbb P(\bigwedge^k\mathbb C^{2N})$. Notice that the condition $k\lneq N$ is necessary, as for $k=N$ one gets the Lagrangian Grassmannian $\LG_{N}$ which is cominuscule, while for $k=1$ one gets the trivial representation.

\paragraph*{Setting.} Let $V\simeq \mathbb C^{2N}$ be a complex vector space endowed with a non-degenerate symplectic form $\boldsymbol{\omega}\in \bigwedge^2V^\vee$. For any $m \in \mathbb Z_{>0}$ we fix the notation
\[ \mathbbm{J}_m := 
{\footnotesize \begin{bmatrix} 
	& & 1 \\
	& \iddots \\
	1 \end{bmatrix}}\in \Mat_{m\times m} \ \ \ \ \ , \ \ \ \ \  
\Omega_m:={\small \begin{bmatrix} 0 & \mathbbm{J}_m \\-\mathbbm{J}_m & 0 \end{bmatrix}}\in \bigwedge^2\mathbb C^{2m} \ . \]
Let $(e_1,\ldots, e_N,e_{-N},\ldots , e_{-1})$ be a basis of $V$ such that the symplectic form $\boldsymbol{\omega}$ is represented by the skew-symmetric matrix $\Omega_N\in \bigwedge^2\mathbb C^{2N}$: in particular, for any $i,j\in [N]$ it holds
\[ \boldsymbol{\omega}(e_i,e_j)=\boldsymbol{\omega}(e_{-i},e_{-j})=0 \ \ \ , \ \ \ \boldsymbol{\omega}(e_i,e_{-j})=\delta_{ij} \ . \]
Given the symplectic group $\Sp^{\boldsymbol{\omega}}(V)=\{A \in \SL(V) \ | \ ^t\!A\Omega_N A=\Omega_N\}$, the isotropic Grassmannian $\IG_{\boldsymbol{\omega}}(k,V)=\left\{ W\subset V \ | \ W \simeq \mathbb C^k, \ W\subset W^\perp\right\}$ is the projective variety of $k$-dimensional linear subspaces of $V$ which are $\boldsymbol{\omega}$--isotropic minimally embedded (via restriction of the Pl\"{u}cker embedding) in $\mathbb P(V_{\omega_k}^{C_N})=\mathbb P(\bigwedge^k\mathbb C^{2N})$. The highest weight vector in $V_{\omega_k}^{C_N}$ and its corresponding $\boldsymbol{\omega}$--isotropic subspace are 
\[ v_{\omega_k}=\bold{e}_{[k]}=e_1\wedge \ldots \wedge e_k \ \ \ , \ \ \ E_k=\langle e_1,\ldots, e_k\rangle_\mathbb C\subset V \ . \] 
Moreover, the $\boldsymbol{\omega}$--orthogonal and the dual subspaces to $E_k$ are respectively
\[ E_k^\perp = \langle e_1,\ldots, e_N, e_{-N}, \ldots, e_{-k-1}\rangle_\mathbb C \ \ \ , \ \ \ E_k^\vee=\langle e_{-k},\ldots, e_{-1}\rangle_\mathbb C \ .\]
For simplicity we write $\Sp_{2N}$ and $\IG(k,2N)$ by omitting the symplectic form $\boldsymbol{\omega}$ and the vector space $V$, and we fix $P:=P_k$ to be the parabolic subgroup defined by the $k$--th simple root $\alpha_k$.

\paragraph*{Levi decomposition of $P$.} The parabolic subgroup $P$ stabilizes the subspace $E_k$. In particular, it admits the Levi decomposition $P=LP^u$ where 
\begin{equation}\label{Levi} 
L=\left\{{\scriptsize \begin{bmatrix} G \\ & S \\& & \mathbbm{J}_k(^t\!R^{-1})\mathbbm{J}_k \end{bmatrix}} \ \bigg| \ {\small \begin{matrix} G \in \GL(E_k) \\ S \in \Sp(E_k^\perp\!/E_k) \end{matrix}}\right\}\simeq \GL(E_k)\times \Sp(E_k^\perp\!/E_k)
\end{equation}
is the Levi factor of $P$ stabilizing the consecutive quotients $E_k, E_k^\perp\!/E_k$ and $E_k^\vee\simeq V/E_k^\perp$ in the flag $E_k\subset E_k^\perp\!\subset V$, and 
\begin{equation}\label{unipotent}
P^u=\left\{ {\scriptsize\begin{bmatrix} I_k & A & B \\ & I_{2N-2k} & \Omega_{N-k}(^t\!A)\mathbbm{J}_k\\ & & I_k \end{bmatrix}} \ \bigg| \ {\footnotesize \begin{matrix} A \ \in \ E_k\otimes E_k^\perp\!/E_k \\ B \in E_k\otimes E_k \\ B\mathbbm{J}_k -\mathbbm{J}_k(^t\!B)=A\Omega_{N-k}(^t\!A)\end{matrix}} \right\}
\end{equation}
is the unipotent radical of $P$ acting trivially on the consecutive quotients of $E_k\subset E_k^\perp\!\subset V$.

\paragraph*{Tangent space.} Consider the tangent space $T_{[v_{\omega_k}]}\IG(k,2N)\simeq \mathfrak{sp}_{2N}/\mathfrak{p}\simeq \mathfrak{p}^u$. The simple root $\alpha_k$ defines a $\mathbb Z$--grading on the Lie algebra $\mathfrak{sp}_{2N}$ whose degrees are bounded by the coefficient $m_{\alpha_k}(\rho)$ of $\alpha_k$ in the longest root $\rho$: since $k \leq N-1$, it holds $m_{\alpha_k}(\rho)=2$, so that $\mathfrak{sp}_{2N}/\mathfrak{p} = \mathfrak{g}_{-1}\oplus \mathfrak{g}_{-2}$. Thus the nilpotent algebra $\mathfrak{p}^u$ is not irreducible as $P$--module, and its only $P$--submodule is $\mathfrak{g}_{-1}$. Deriving the description of the unipotent radical $P^u$ in \eqref{unipotent} we get
\[ \mathfrak{p}^u= \big( E_k^\vee \otimes E_k^\perp\!/\!E_k \big) \oplus \Sym^2\!E_k^\vee \ , \]
whose only $P$--invariant (for the adjoint action of $P$) summand is $\mathfrak{g}_{-1}=E_k^{\vee} \otimes E_k^\perp\!/\!E_k$.

\subsection{$P_k$--orbit poset in $\mathfrak{sp}_{2N}/\mathfrak{p}_k$, and orbit dimensions}

We want to determine the poset of $P$--orbits in the tangent space $T_{[v_{\omega_k}]}\IG(k,2N)\simeq \mathfrak{sp}_{2N}/\mathfrak{p}$, that is in the nilpotent algebra $\mathfrak{p}^u= \Sym^2\!E_k^\vee \oplus \big( E_k^\vee \otimes E_k^\perp\!/\!E_k \big)$. Since $P$ has nilpotency class $\ell(P)=2$ (cf. \eqref{eq:nilpotency class}), we already know from \cite[Theorem 1.1]{hillerohrle} that $\mathfrak{p}^u$ has finitely many $P$--orbits. An element of the tangent space is of the form
\[ \sigma + H \ \in \ \Sym^2\!E_k^\vee \oplus \left( E_k^\vee \otimes E_k^\perp\!/E_k\right) \ . \] 
The action of $\GL(E_k)\subset L$ conjugates it to one of the form $(e_{-k}^2+\ldots + e_{-k+r-1}^2)+\widehat{H}$, where $r$ is the rank of the quadratic form $\sigma \in \Sym^2\!E_k^\vee$. In this respect, we may assume $\sigma$ to be
\[ \sigma_r:=e_{-k}^2+\ldots + e_{-k+r-1}^2 \ . \]
In the following we describe separately the action of the Levi factor $L$ and the unipotent radical $P^u$ on an element of the form $\sigma_r+H$ for fixed $r\leq k$ and $H \in E_k^\vee \otimes E_k^\perp\!/E_k$.

\paragraph*{Action of $P^u$.} Consider the unipotent radical $P^u\subset P$ in \eqref{unipotent}, which depends on the entries $A \in E_k\otimes E_k^\perp\!/E_k$ and $B \in E_k\otimes E_k$ only. Notice that the action of $B$ on $\Sym^2\!E_k^\vee$ is identically zero since there is no non-zero projection from $(E_k\otimes E_k)\otimes \Sym^2\!E_k^\vee$ onto $\Sym^2\!E_k^\vee\oplus \big (E_k^\vee \otimes E_k^\perp\!/E_k \big)$. Similarly, also the actions of $A$ and $B$ on $E_k^\vee \otimes E_k^\perp\!/E_k$ are identically zero. \\
\indent The only non-trivial action is the one of $A$ on $\Sym^2\!E_k^\vee$, which by Schur's theorem coincides with the contraction map
\[ \begin{matrix}
\left( E_k\otimes E_k^\perp\!/E_k \right) \otimes \Sym^2\!E_k^\vee & \longrightarrow & E_k^\vee \otimes E_k^\perp\!/E_k \\
\left( \sum_{i,j} a_{ij}e_i\otimes e_j \right) \otimes f & \mapsto & \sum_{i,j} a_{ij}\frac{\partial{f}}{\partial e_{-i}}\otimes e_j
\end{matrix} \ . \]
In particular, given $A=\sum_{i=1}^k\sum_{j=k+1}^{-k-1}a_{ij}e_i\otimes e_j \in E_k \otimes E_k^\perp\!/E_k$, one gets 
\[ A\cdot \sigma_r = \sum_j\sum_{i=k-r+1}^k 2a_{ij}e_{-i}\otimes e_j \ . \]

\begin{obs*}
    Here is the big difference between the cominuscule and the non-cominuscule cases. In the former, the unipotent radical is abelian, hence it acts trivially on the tangent space. On the contrary, in the non-cominuscule case the unipotent radical acts non-trivially.
    \end{obs*}

It follows that, for any unipotent element $g_A\in P^u$ depending on $A\in E_k\otimes E_k^\perp\!/E_k$, and for any $H=\sum_{i,j}h_{ij}e_{-i}\otimes e_j\in E_k^\vee \otimes E_k^\perp\!/E_k$, it holds
\begin{align*}
g_A \cdot (\sigma_r + H) & = \left(I_k\cdot \sigma_r\right) + \left(I_{2N-2k}\cdot H\right) + \left(A\cdot \sigma_r\right) = \sigma_r + H + A\cdot \sigma_r \\
& = \sigma_r + \sum_j \left[ \sum_{i=1}^{k-r}h_{ij}e_{-i}\otimes e_j + \sum_{i=k-r+1}^k(h_{ij}+2a_{ij})e_{-i}\otimes e_j \right] \ . 
\end{align*}
Therefore acting by $P^u$ allows to \textquotedblleft truncate\textquotedblright \space the summand $H$. We conclude that there are infinitely many $P^u$--orbits in $\sigma_r + \left(E_k^\vee \otimes E_k^\perp\!/E_k\right)$, namely
\[ \sigma_r + \left(E_k^\vee \otimes E_k^\perp\!/E_k\right) \ = \ \bigsqcup_{Q\in \mathbb C^{k-r}\otimes \mathbb C^{2N-2k}} P^u\cdot \left(\sigma_r + \sum_{i=1}^{k-r}e_{-i}\otimes Q^i\right) \ , \]
\noindent where $Q^i$ denotes the $i$-th column of the matrix $Q \in \mathbb C^{k-r}\otimes \mathbb C^{2N-2k}$.

\begin{obs*}
	Be aware that to be infinitely--many are the $P^u$--orbits (not $P$--orbits) for a fixed $\sigma \in \Sym^2\!E_k^\vee$. We are going to see that the additional action of the Levi factor $L$ reduces the number of orbits in the whole tangent space to finitely many.
\end{obs*}

\paragraph*{The action of $L$.} Consider the Levi factor $L\subset P$ in \eqref{Levi}. In light of the previous arguments, up to acting by both $\GL(E_k)\subset L$ and $P^u$, we study the action of $L$ on an element of the form
\begin{equation}\label{s+H after P^u} 
\sigma_r + H \ = \ \left(e_{-k}^2 +\ldots + e_{-k+r-1}^2\right) + \sum_{i=1}^{k-r}e_{-i}\otimes q_i 
\end{equation}
for certain vectors $q_i \in E_k^\perp\!/E_k$. Also, we are left with considering the action of the stabilizer
\[ \Stab_{L}(\sigma_r)= \left\{ {\scriptsize \begin{bmatrix} \ast & \ast \\ & \ast \\ & & N \\ & & & O & R \\ & & & & M \end{bmatrix}}\in L \ \bigg| \ {\footnotesize \begin{matrix} O \in \Orth(r)\\M \in \GL_{k-r} \\ R \in \mathbb C^{r}\otimes \mathbb C^{k-r} \\ N \in \Sp(E_k^\perp\!/E_k) \end{matrix}}  \right\} \ \simeq \ \Stab_{\GL(E_k)}(\sigma_r)\times \Sp(E_k^\perp\!/E_k) \ , \]
where the $k\times k$ block with asterisks at the top uniquely depends on the $k\times k$ block ${\tiny \begin{bmatrix} O & R \\ & M \end{bmatrix}}$.\\
\indent Let $\langle q_1,\ldots, q_{k-r}\rangle_{\mathbb C}\subset E_k^\perp\!/E_k$ be the subspace spanned by the second tensor-entries in $H$ \eqref{s+H after P^u}, that is the image of the linear map $H: E_k \rightarrow E_k^\perp\!/E_k$. Fix 
\[ h:=\rk(H) \leq \min\{ 2N-2k, k-r\} \ . \] 
Acting by a proper permutation matrix in $\GL_{k-r}\subset \Stab_{\GL(E_k)}(\sigma_r)$, we can reorder the summands in $H$ such that $(q_1,\ldots, q_h)$ is a basis of $\im(H)$. Then for any $h+1 \leq i \leq k-r$ we can write $q_i=\sum_{j=1}^h\alpha_{ij}q_j$ with respect to this basis, for a suitable matrix $(\alpha_{ij})\in \mathbb C^{k-r-h}\otimes \mathbb C^{h}$, getting
\[ \sigma_r + H = \sigma_r + \sum_{i=1}^h \left(e_{-i} +  \sum_{\ell=h+1}^{k-r}\alpha_{\ell i}e_{-\ell} \right)\otimes q_i \ . \]
\noindent Moreover, by applying a basis change via $\GL_{k-r}\subset \Stab_{\GL(E_k)}\subset L$ in the above first tensor-entries, we can move such $\sigma_r+H$ to the point $\sigma_r + \sum_{i=1}^h e_{-i} \otimes q_i$. Thus for any $h$--dimensional subspace $Q=\langle q_1,\ldots, q_h\rangle_\mathbb C \in \Gr(h,E_k^\perp\!/E_k)$ we set
\begin{equation}\label{s+H after P^u with h}
\sigma_r + H_Q := \sigma_r + \sum_{i=1}^h e_{-i} \otimes q_i \ .
\end{equation}

\begin{obs}\label{rmk:h well defined as invariant}
	The dimension $h$ is an invariant for the $P^u$--action too (hence for the $P$--action). Indeed, given $\supp(\sigma_r)=\langle x \in E_k \ | \ \sigma_r(x)\neq 0\rangle_\mathbb C=\langle e_{k-r+1},\ldots, e_{k}\rangle_\mathbb C\subset E_k$, one has
	\[ P^u\cdot (\sigma_r+H_Q) = \{ \sigma_r + H_Q + H' \ | \ H' \in \supp(\sigma_r)^\vee\otimes E_k^\perp\!/E_k\} \ . \]
	Then $h$ is well defined for any $\sigma \in \Sym^2E_k^\vee$ as the invariant $h:=\min\{\rk(Q|Q') \ | \ H_{Q'}\in \supp(\sigma)^\vee \otimes E_k^\perp\!/E_k\}$, or equivalently as
 \[ h:= \rk\left( H _{|_{E_k/\supp(\sigma)}}\right) \leq \min\{ 2N-2k , k-\rk(\sigma)\} \ . \]
\end{obs}

\indent Up to now, we have used the action of both $L$ and $P^u$ on $\sigma_r+H$ for minimizing the number of summands appearing in $H$. The next step is to find representatives for such \textquotedblleft minimal\textquotedblright \space summands and we do so via the action of $\Sp(E_k^\perp\!/E_k)$.\\
\indent Let $\boldsymbol{\omega'}\in \bigwedge^2\!(E_k^\perp\!/E_k)^\vee$ be the non-degenerate symplectic form obtained as restriction of $\boldsymbol{\omega}\in \bigwedge^2V^\vee$. It is represented by the skew-symmetric matrix $\Omega_{N-k}\in \bigwedge^2\mathbb C^{2N-2k}$. \\
For any subspace $Q \in \Gr(h,E_k^\perp\!/E_k)$, the restriction of $\boldsymbol{\omega'}$ to $Q$ has rank 
\[ \rk\!\big(\boldsymbol{\omega'}_{|_{Q}}\big)=h-\dim (Q\cap Q^\perp)\equiv 0 \ \text{(mod $2$)} \ . \] Notice that: if $Q \subset Q^\perp$, i.e. $Q \in \IG_{\boldsymbol{\omega'}}(h,E_k^\perp\!/E_k)$ and $\rk\!\big( \boldsymbol{\omega'}_{|_{Q}}\big)=0$, then $h\leq N-k$ and the symplectic group $\Sp(E_k^\perp\!/E_k)$ conjugates $\sigma_r+H_Q$ to the point $\sigma_r + \sum_{i=1}^{h}e_{-i}\otimes e_{k+i}$; if $Q\cap Q^\perp=\{0\}$, then $Q\oplus Q^\perp =E_k^\perp\!/E_k$ and $\rk\!\big( \boldsymbol{\omega'}_{_Q})=h$ is even, and the symplectic group $\Sp(E_k^\perp\!/E_k)$ conjugates $\sigma_r+ H_Q$ to the point $\sigma_r + \sum_{i=1}^{\frac{h}{2}}e_{-i}\otimes e_{k+i} + \sum_{i=1}^{\frac{h}{2}}e_{-\frac{h}{2}-i}\otimes e_{-k-i}$.\\
\indent In general, given $t:= h- \rk\!\big( \boldsymbol{\omega'}_{|_Q} \big)$ --hence $h\equiv t$ (mod $2$)--, it holds $Q\cap Q^\perp \in \IG_{\boldsymbol{\omega'}}(t,E_k^\perp\!/E_k)$ and $(Q/(Q\cap Q^\perp))\oplus (Q^\perp\!/(Q\cap Q^\perp))=(Q+Q^\perp)/(Q\cap Q^\perp)$. Thus the action of $\Sp(E_k^\perp\!/E_k)$ conjugates the subspace $H_Q$ in \eqref{s+H after P^u with h} to the subspace
\begin{equation}\label{representative of P-orbits}
H_{(h,t)} := \sum_{i=1}^t e_{-i}\otimes e_{k+i} + \sum_{i=1}^{\frac{h-t}{2}} e_{-t-i}\otimes e_{k+t+i} + \sum_{i=1}^{\frac{h-t}{2}} e_{-t-\frac{h-t}{2}-i}\otimes e_{-k-t-i} \ , 
\end{equation}
and the point $\sigma_r+H_Q$ to the point $\sigma_r+H_{(h,t)}$. This proves the following result.

\begin{teo}\label{thm:orbits in IG}
	For any $2\leq k \leq N-1$, the nilpotent algebra $\mathfrak{p}_k^u$ splits in the finitely many $P_k$--orbits 
	\begin{align}\label{def:orbit}
	\cal O_{(r,h,t)} & :=\left\{ \sigma + H \in \Sym^2\!E_k^\vee \oplus \left(E_k^\vee \otimes E_k^\perp\!/E_k \right) \ \bigg| \ {\footnotesize \begin{matrix} 
        r=\rk(\sigma) \ \ , \ \ 
        h=\rk\!\left(H_{|_{E_k/\supp(\sigma)}}\right) \\
		t=h-\rk\!\left(\boldsymbol{\omega}_{\big|{\im\!\big(H_{|_{E_k/\supp(\sigma)}}\big)}}\right)
        \end{matrix}} \right\}\\ \nonumber
    & \ = P_k \cdot \left(\sigma_r + H_{(h,t)}\right) 
	\end{align}
	where $\supp(\sigma):=\{x \in E_k \ | \ \sigma(x)\neq 0\}\subset E_k$ is the support of the quadratic form $\sigma \in \Sym^2E_k^\vee$, and $\sigma_r+H_{(h,t)}$ is the representative in \eqref{representative of P-orbits}. 
\end{teo}

\begin{obs*}
    We stress out that the triplet of invariants $(r,h,t)$ satisfies the following inequalities:
    \begin{equation}\label{constraints on invariants}
r \leq k \ \ \ , \ \ \ h \leq \min\{k-r,2N-2k\} \ \ \ , \ \ \ t \leq \min\{h,N-k\} \ \ \ , \ \ \ t \equiv h \ \text{(mod $2$)} \ . 
\end{equation}
    \end{obs*}

\subsection{Inclusions among orbit closures}

Next we show that the poset graph in \autoref{figure:graph cominuscule} does not applying to the secant variety of lines to an isotropic Grassmannian $\IG(k,2N)$ by proving that the tangent $\Sp_{2N}$--orbits are not totally ordered: by homogeneity of $\IG(k,2N)$, it is enough to prove that the $P_k$--orbits in the tangent space $T_{[v_{\omega_k}]}\IG(k,2N)$ (the ones in Theorem \ref{thm:orbits in IG}) are not so.

\begin{prop}\label{prop:inclusions IG}
	The inclusions among the closures of the $P_k$--orbits $\cal O_{(r,h,t)}$ in $\mathfrak{p}_k^u$ are ruled by the degeneracies of the ranks in the matrix spaces
	\[ \Sym^2\!E_k^\vee \ \ \ , \ \ \ \left(\frac{E_k}{\supp(\sigma)}\right)^\vee\otimes (E_k^\perp\!/E_k) \ \ \ , \ \ \ \bigwedge^2\left(\frac{(E_k/\supp(\sigma))^\perp}{E_k/\supp(\sigma)}\right)^\vee \ . \]
	More precisely, for any triplet $(r,h,t)$ satisfying \eqref{constraints on invariants}, the minimal inclusions are:
	\begin{enumerate}
	\item[$i)$] $\cal O_{(r,h,t)}\subset \overline{\cal O_{(r,h+1,t+1)}}$ ,
    \item[$ii)$] $\cal O_{(r,h,t)}\subset \overline{\cal O_{(r,h+1,t-1)}}$ ,
	\item[$iii)$] $\cal O_{(r,h,t)}\subset \overline{\cal O_{(r+1,h-1,t+1)}}$ ,
    \item[$iv)$] $\cal O_{(r,h,t)}\subset \overline{\cal O_{(r+1,h-1,t-1)}}$ ,
	\end{enumerate}
    assumed that all the above triplets satisfy the constraints \eqref{constraints on invariants} as well. In particular, for any triplet of invariants $(r,h,t)\neq (k,0,0)$ it holds $\cal O_{(r,h,t)}\subset \overline{\cal O_{(k-1,1,1)}}$.
\end{prop}

\begin{cor}
	The orbit $\cal O_{(k,0,0)}=P_k\cdot \sigma_k$ is dense in $\mathfrak{p}_k^u$, and its complement
	\[\overline{\cal O_{(k-1,1,1)}} = \mathfrak{p}_k^u\setminus \cal O_{(k,0,0)}\] 
	is the hypersurface defined by the vanishing of the determinant in $\Sym^2\!E_k^\vee$. 
\end{cor}

\begin{figure}[h]
\begin{center}
\begin{subfigure}[l]{0.45\linewidth}
        \begin{center}
\begin{tikzpicture}[scale=1.5,tdplot_rotated_coords,
                    rotated axis/.style={->,purple,ultra thick},
                    blackBall/.style={ball color = black!20},
                    cyanBall/.style={ball color = cyan},
                    redBall/.style={ball color = red},
                    blueBall/.style={ball color = blue},
                    borderBall/.style={ball color = cyan,opacity=.35},
                    redborderBall/.style={ball color = red,opacity=0.35},
                    yellowBall/.style={ball color = yellow},
                    greenBall/.style={ball color = green},
                    very thick]

\foreach \x in {0,1,2}
   \foreach \y in {0,1,2}
      \foreach \z in {0,1,2}{
           \ifthenelse{  \lengthtest{\x pt < 2pt}  }{
             \draw[dashed, line width=0.5pt, color=gray] (\x,\y,\z) -- (\x+1,\y,\z);
             \shade[rotated axis,blackBall] (\x,\y,\z) circle (0.025cm); 
           }{}
           \ifthenelse{  \lengthtest{\y pt < 2pt}  }{
               \draw[dashed, line width=0.5pt, color=gray] (\x,\y,\z) -- (\x,\y+1,\z);
               \shade[rotated axis,blackBall] (\x,\y,\z) circle (0.025cm);
           }{}
           \ifthenelse{  \lengthtest{\z pt < 2pt}  }{
               \draw[dashed, line width=0.5pt, color=gray] (\x,\y,\z) -- (\x,\y,\z+1);
               \shade[rotated axis,blackBall] (\x,\y,\z) circle (0.025cm);
           }{}
}

\shade[rotated axis,yellowBall] (0,0,0) circle (0.05cm); 
\shade[rotated axis,blackBall] (1,0,0) circle (0.05cm); 
\shade[rotated axis,yellowBall] (2,0,0) circle (0.05cm);
\shade[rotated axis,blackBall] (0,1,0) circle (0.05cm); 
\shade[rotated axis,yellowBall] (1,1,0) circle (0.05cm); 
\shade[rotated axis,blackBall] (2,1,0) circle (0.05cm);
\shade[rotated axis,yellowBall] (0,2,0) circle (0.05cm); 
\shade[rotated axis,blackBall] (1,2,0) circle (0.05cm); 
\shade[rotated axis,yellowBall] (2,2,0) circle (0.05cm);

\shade[rotated axis,yellowBall] (0,0,1) circle (0.05cm); 
\shade[rotated axis,blackBall] (1,0,1) circle (0.05cm); 
\shade[rotated axis,cyanBall] (2,0,1) circle (0.05cm);
\shade[rotated axis,blackBall] (0,1,1) circle (0.05cm); 
\shade[rotated axis,redBall] (1,1,1) circle (0.05cm); 
\shade[rotated axis,blackBall] (2,1,1) circle (0.05cm);
\shade[rotated axis,yellowBall] (0,2,1) circle (0.05cm); 
\shade[rotated axis,blackBall] (1,2,1) circle (0.05cm); 
\shade[rotated axis,cyanBall] (2,2,1) circle (0.05cm);

\shade[rotated axis,cyanBall] (0,0,2) circle (0.05cm); 
\shade[rotated axis,blackBall] (1,0,2) circle (0.05cm); 
\shade[rotated axis,blueBall] (2,0,2) circle (0.05cm);
\shade[rotated axis,blackBall] (0,1,2) circle (0.05cm); 
\shade[rotated axis,blueBall] (1,1,2) circle (0.05cm); 
\shade[rotated axis,blackBall] (2,1,2) circle (0.05cm);
\shade[rotated axis,cyanBall] (0,2,2) circle (0.05cm); 
\shade[rotated axis,blackBall] (1,2,2) circle (0.05cm); 
\shade[rotated axis,blueBall] (2,2,2) circle (0.05cm);

\shade[rotated axis,borderBall] (2,0,1) circle (0.1cm);
\shade[rotated axis,borderBall] (2,2,1) circle (0.1cm);
\shade[rotated axis,borderBall] (0,2,2) circle (0.1cm);
\shade[rotated axis,borderBall] (0,0,2) circle (0.1cm);
\shade[rotated axis,redborderBall] (1,1,1) circle (0.1cm);

\draw (1,1,2) edge [->, shorten >= 8pt, shorten <= 8pt,color=blue!30] node [left] {} (2,0,1);
\draw (2,0,1) edge [->,shorten >= 8pt, shorten <= 8pt,color=cyan!50] node [left] {} (1,1,1);
\draw (0,2,2) edge [->,shorten >= 8pt, shorten <= 8pt,color=cyan!50] node [left] {} (1,1,1);
\draw (0,0,2) edge [->,shorten >= 8pt, shorten <= 8pt,color=cyan!50] node [left] {} (1,1,1);
\draw (2,2,1) edge [->,shorten >= 8pt, shorten <= 8pt,color=cyan!50] node [left] {} (1,1,1);
\draw (2,0,2) edge [->,shorten >= 8pt, shorten <= 8pt,color=blue!30] node [left] {} (1,1,2);
\draw (2,2,2) edge [->,shorten >= 8pt, shorten <= 8pt,color=blue!30] node [left] {} (1,1,2);

\draw (2,0,0.85) node[scale=0.8, below] {\textcolor{cyan}{$(0,1,-1)$}};
\draw (1.3,1,1.0) node[scale=0.8, below] {\textcolor{red}{$(0,0,0)$}};
\draw (2,2.5,1.1) node[scale=0.8, below] {\textcolor{cyan}{$(0,1,1)$}};
\draw (0.05,-0.6,2) node[scale=0.8, below] {\textcolor{cyan}{$(1,-1,-1)$}};
\draw (2,-0,1.85) node[scale=0.8, below] {\textcolor{blue}{$(1,1,-1)$}};
\draw (0.3,0.6,2) node[scale=0.8, below] {\textcolor{blue}{$(1,0,0)$}};
\draw (2,2.5,2.1) node[scale=0.8, below] {\textcolor{blue}{$(1,1,1)$}};
\draw (0.05,2,2.4) node[scale=0.8, below] {\textcolor{cyan}{$(1,-1,1)$}};

\end{tikzpicture}
\end{center}
\caption{}
\label{fig:lattice orbit inclusions.a}

\end{subfigure}
\hfill
\begin{subfigure}[r]{0.5\linewidth}
        \begin{center}
\begin{tikzpicture}[scale=1.5,tdplot_rotated_coords,
                    rotated axis/.style={->,purple,ultra thick},
                    blackBall/.style={ball color = black!20},
                    cyanBall/.style={ball color = cyan},
                    redBall/.style={ball color = red},
                    blueBall/.style={ball color = blue},
                    borderBall/.style={ball color = cyan,opacity=.35},
                    redborderBall/.style={ball color = red,opacity=0.35},
                    yellowborderBall/.style={ball color = yellow,opacity=.35},
                    yellowBall/.style={ball color = yellow},
                    greenBall/.style={ball color = green},
                    very thick]

\foreach \x in {0,1,2}
   \foreach \y in {0,1,2}
      \foreach \z in {0,1,2}{
           \ifthenelse{  \lengthtest{\x pt < 2pt}  }{
             \draw[dashed, line width=0.5pt, color=gray] (\x,\y,\z) -- (\x+1,\y,\z);
             \shade[rotated axis,blackBall] (\x,\y,\z) circle (0.025cm); 
           }{}
           \ifthenelse{  \lengthtest{\y pt < 2pt}  }{
               \draw[dashed, line width=0.5pt, color=gray] (\x,\y,\z) -- (\x,\y+1,\z);
               \shade[rotated axis,blackBall] (\x,\y,\z) circle (0.025cm);
           }{}
           \ifthenelse{  \lengthtest{\z pt < 2pt}  }{
               \draw[dashed, line width=0.5pt, color=gray] (\x,\y,\z) -- (\x,\y,\z+1);
               \shade[rotated axis,blackBall] (\x,\y,\z) circle (0.025cm);
           }{}
}

\shade[rotated axis,yellowBall] (0,0,0) circle (0.05cm); 
\shade[rotated axis,blackBall] (1,0,0) circle (0.05cm); 
\shade[rotated axis,yellowBall] (2,0,0) circle (0.05cm);
\shade[rotated axis,blackBall] (0,1,0) circle (0.05cm); 
\shade[rotated axis,yellowBall] (1,1,0) circle (0.05cm); 
\shade[rotated axis,blackBall] (2,1,0) circle (0.05cm);
\shade[rotated axis,yellowBall] (0,2,0) circle (0.05cm); 
\shade[rotated axis,blackBall] (1,2,0) circle (0.05cm); 
\shade[rotated axis,yellowBall] (2,2,0) circle (0.05cm);

\shade[rotated axis,yellowBall] (0,0,1) circle (0.05cm); 
\shade[rotated axis,blackBall] (1,0,1) circle (0.05cm); 
\shade[rotated axis,cyanBall] (2,0,1) circle (0.05cm);
\shade[rotated axis,blackBall] (0,1,1) circle (0.05cm); 
\shade[rotated axis,redBall] (1,1,1) circle (0.05cm); 
\shade[rotated axis,blackBall] (2,1,1) circle (0.05cm);
\shade[rotated axis,yellowBall] (0,2,1) circle (0.05cm); 
\shade[rotated axis,blackBall] (1,2,1) circle (0.05cm); 
\shade[rotated axis,cyanBall] (2,2,1) circle (0.05cm);

\shade[rotated axis,cyanBall] (0,0,2) circle (0.05cm); 
\shade[rotated axis,blackBall] (1,0,2) circle (0.05cm); 
\shade[rotated axis,blueBall] (2,0,2) circle (0.05cm);
\shade[rotated axis,blackBall] (0,1,2) circle (0.05cm); 
\shade[rotated axis,blueBall] (1,1,2) circle (0.05cm); 
\shade[rotated axis,blackBall] (2,1,2) circle (0.05cm);
\shade[rotated axis,cyanBall] (0,2,2) circle (0.05cm); 
\shade[rotated axis,blackBall] (1,2,2) circle (0.05cm); 
\shade[rotated axis,blueBall] (2,2,2) circle (0.05cm);

\shade[rotated axis,yellowborderBall] (0,0,1) circle (0.1cm);
\shade[rotated axis,yellowborderBall] (0,2,1) circle (0.1cm);
\shade[rotated axis,yellowborderBall] (2,0,0) circle (0.1cm);
\shade[rotated axis,yellowborderBall] (2,2,0) circle (0.1cm);
\shade[rotated axis,redborderBall] (1,1,1) circle (0.1cm);

\draw (1,1,1) edge [->,shorten >= 8pt, shorten <= 8pt,color=green] node [left] {} (0,0,1);
\draw (1,1,1) edge [->,shorten >= 8pt, shorten <= 8pt,color=green] node [left] {} (0,2,1);
\draw (1,1,1) edge [->,shorten >= 8pt, shorten <= 8pt,color=green] node [left] {} (2,0,0);
\draw (1,1,1) edge [->,shorten >= 8pt, shorten <= 8pt,color=green] node [left] {} (2,2,0);
\draw (2,2,0) edge [->,shorten >= 8pt, shorten <= 8pt,color=yellow] node [left] {} (1,1,0);
\draw (1,1,0) edge [->,shorten >= 8pt, shorten <= 8pt,color=yellow] node [left] {} (0,0,0);
\draw (2,0,0) edge [->,shorten >= 8pt, shorten <= 8pt,color=yellow] node [left] {} (1,1,0);
\draw (1,1,0) edge [->,shorten >= 8pt, shorten <= 8pt,color=yellow] node [left] {} (0,2,0);

\draw (0,-0.7,1) node[scale=0.8, below] {\textcolor{green}{$(0,-1,-1)$}};
\draw (0,2,1.4) node[scale=0.8, below] {\textcolor{green}{$(0,-1,1)$}};
\draw (2,0.1,-0.1) node[scale=0.8, below] {\textcolor{green}{$(-1,1,-1)$}};
\draw (1.4,2.2,0) node[scale=0.8, below] {\textcolor{green}{$(-1,1,1)$}};

\end{tikzpicture}
\end{center}
\caption{}
\label{fig:lattice orbit inclusions.b}
\end{subfigure}

\end{center}
\caption{The orbit inclusions are described in the above lattice cube, centered at the red point $(0,0,0)$ corresponding to a fixed triplet of invariants $(r,h,t)$. Any other point $(a,b,c)$ in the lattice corresponds to the triplet (not necessarily of invariants) $(r+a,h+b,t+c)$. The gray points are triplets which are not of invariants (constraints \eqref{constraints on invariants} are not satisfied). \autoref{fig:lattice orbit inclusions.a} shows the orbits in the lattice degenerating to $\cal O_{(r,h,t)}$: the cyan arrows denote the minimal degenerations. \autoref{fig:lattice orbit inclusions.b} shows the orbits in the lattice to which $\cal O_{(r,h,t)}$ degenerates: the green arrows denote the minimal degenerations. }
\label{fig:lattice orbit inclusions}
\end{figure}

\begin{proof}[Proof to Proposition \ref{prop:inclusions IG}]
Fix an orbit $\cal O_{(r,h,t)}$. Note that the invariant $r$ (as rank in $\Sym^2E_k^\vee$) can only decrease during orbit degenerations. However this is not true for the values $h,t$ since they can also increase after the $P^u$--action. \\
\indent First, we analyse the possible degenerations to $\cal O_{(r,h,t)}$ from orbits $\cal O_{(r+a,h+b,t+c)}$ such that $a \in \{0,1\}$ and $b,c\in\{-1,0,1\}$: these correspond to the lattice points in \autoref{fig:lattice orbit inclusions}. For any $b+c \equiv 1$ (mod $2$) the corresponding triplet does not satisfy the constraints \eqref{constraints on invariants}, hence it does not define an orbit (cf. gray points in the figure). Next, we focus on the candidates for minimal degenerations (cf. cyan points in \autoref{fig:lattice orbit inclusions.a}).
\begin{enumerate}
\item[$i)$] The sequence $\alpha(\epsilon) := \sigma_r+H_{(h,t)}+\frac{1}{\epsilon}e_{-h-1}\otimes e_{k+h+1}$ lies in $\cal O_{(r,h+1,t+1)}$ and has limit $\sigma_r+H_{(h,t)}$ for $\epsilon \to \infty$.
\item[$ii)$] The sequence $\beta(\epsilon):=\sigma_r + H_{(h+1,t-1)} + \big(\frac{1}{\epsilon}-1 \big)e_{-t-\frac{h-t+2}{2}}\otimes e_{-k-t}$ lies in $\cal O_{(r,h+1,t-1)}$ and has limit $\sigma_r + H_{(h,t)}$ for $\epsilon \to \infty$.
\item[$iii)$] The sequence 
	$\gamma(\epsilon):= \sigma_r + \frac{1}{\epsilon}e_{-h}^2 + H_{(h,t)}$ lies in $\cal O_{(r+1,h-1,t+1)}$ (since the $P^u$-action allows to annihilate the summand $e_{-h}\otimes e_{-k-t-\frac{h-t}{2}}$ in $H_{(h,t)}$) and has limit $\sigma_r+H_{(h,t)}$ for $\epsilon \to \infty$.
 \item[$iv)$] The sequence $\delta(\epsilon):= \sigma_{r+1} + \left(\frac{1}{\epsilon} -1 \right) e_{-k}^2 +H_{(h-1,t-1)} + e_{-k}\otimes e_{-k-t}$ lies in $\cal O_{(r+1,h-1,t-1)}$ (since the $P^u$-action allows to annihilate the summand $e_{-k}\otimes e_{-k-t}$) and it has limit $e_{-k+1}^2+ \ldots + e_{-k+r}^2 + H_{(h-1,t-1)} + e_{-k}\otimes e_{-k-t}$ lying in $\cal O_{(r,h,t)}$.
\end{enumerate}
\noindent Combining the degenerations in $i),ii),iii),iv)$ one also gets the degenerations both from the blue points in \autoref{fig:lattice orbit inclusions.a}, and to the yellow points in \autoref{fig:lattice orbit inclusions.b}.\\
\indent Second, we observe that the degenerations of all the other orbits outside the cube in \autoref{fig:lattice orbit inclusions} can also be obtained from the ones above: the degenerations $i),ii)$ allow to move on the plane $(h,t)$ for a fixed $r$, while the degenerations $iii),iv)$ allow to go down with the rank $r$. This also proves that $i),ii),iii),iv)$ are minimal. \\
\indent Finally, we prove that the orbit $\cal O_{(k-1,1,1)}$ degenerates to any other orbit but $\cal O_{(k,0,0)}$: this is straightforward for $\cal O_{(r,0,0)}$, while for $h\geq 1$ the sequence $\eta(\epsilon):= \sigma_r + \frac{1}{\epsilon}\big( e_{-k+r}^2 + \ldots + e_{-2}^2 \big) + H_{(h,t)}$ lies in $\cal O_{(k-1,1,1)}$ (the $P^u$-action conjugates $H_{(h,t)}$ to $e_{-1}\otimes e_{k+1}$) and has limit $\sigma_r+ H_{(h,t)}$. \\
\end{proof}

\subsection{Dimensions of the orbits} 

We conclude the study of the $P_k$-orbits in $\mathfrak{p}_k^u$ by computing their dimensions. From the description of $\mathfrak{p}^u\simeq \Sym^2E_k^\vee \otimes E_k^\perp\!/E_k$ we know that  
\[ \dim \IG_{\boldsymbol{\omega}}(k,2N)=\dim \mathfrak{p}_k^u =\frac{k(k+1)}{2}+k(2N-2k) = k\left( \frac{4N-3k+1}{2} \right) \ . \]

\begin{obs}\label{rmk:dim IG}
	One can recover the above dimension also by noticing that $\IG_{\boldsymbol{\omega}}(k,2N)$ is the kernel of the section $s_{\boldsymbol{\omega}} \in H^0\big(\Gr(k,2N),\bigwedge^2\cal U\big)$ corresponding to the symplectic form $\boldsymbol{\omega}\in \bigwedge^2(\mathbb C^{2N})^\vee$ (via Borel--Weil's Theorem). Thus the fiber dimension theorem implies
	\[\dim \IG_{\boldsymbol{\omega}}(k,2N)=\dim \Gr(k,2N)-\rk \left(\bigwedge^2\cal U\right)=k\left( \frac{4N-3k+1}{2} \right) \ . \]
\end{obs}

\begin{prop}\label{prop:dimensions IG}
	In the previous notation, the orbit $\cal O_{(r,h,t)}$ has dimension
	\begin{equation}\label{dimensions of P-orbits}
	\dim \cal O_{(r,h,t)}= \frac{r(r+1)}{2} + (r+h)(2N-k-r) + \frac{t(t+1)}{2}-h^2-t^2 \ .
	\end{equation}
\end{prop}
\begin{proof}
	For any $h,t$ satisfying conditions \eqref{constraints on invariants}, consider the incidence variety 
	\[ \cal I_{h,t}:=\left\{(Q',Q)\in \IG_{\boldsymbol{\omega}'}(t,E_k^\perp\!/E_k)\times \Gr(h,E_k^\perp\!/E_k) \ | \ Q'=Q\cap Q^\perp \right \}\]
	whose projection onto the first factor $\pi: \cal I_{h,t}\rightarrow \IG_{\boldsymbol{\omega}'}(t,E_k^\perp\!/E_k)$ has fiber 
	\begin{align*}
	\pi^{-1}(W) \simeq \left\{ Q \in \Gr(h,E_k^\perp\!/E_k) \ | \ Q\cap Q^\perp=W\right\} \simeq \left\{ \tilde{Q} \in \Gr(h-t,W^\perp\!/W) \ | \ \tilde{Q}\cap \tilde{Q}^\perp=\{0\}\right\} \ . 
	\end{align*}
	As the latter set is a dense in the Grassmannian $\Gr(h-t,W^\perp\!/W)$, from the fiber dimension theorem we get
	\begin{align*}
	\dim \cal I_{h,t} & = \dim \IG_{\boldsymbol{\omega}'}(t,E_k^\perp\!/E_k) + \dim \Gr(h-t,2N-2k-2t) \\
	& = \frac{t(t+1)}{2} + t(2N-2k-2t) + (h-t)(2N-2k-h-t) \ .
	\end{align*}
	\indent Let $\big[ \Sym^2\mathbb C^k \big]_r$ be the set of $k\times k$ symmetric matrices of rank $r$, having dimension $\frac{r(r+1)}{2}+r(k-r)$. For any triplet of invariants $(r,h,t)$ we consider the fibration
	\[ \begin{matrix} \rho : & \cal O_{(r,h,t)} & \longrightarrow & \left[\Sym^2\!E_k^\vee\right]_r \times \cal I_{h,t} \\
	& \sigma + H & \mapsto & \left( \ \sigma \ , \ [ \ \im(H)\cap \im(H)^\perp \subset \im(H) \ ] \ \right) \end{matrix} \ . \]
	Given $U:=\langle e_{-k+r}, \ldots, e_{-1}\rangle_\mathbb C\subset E_k^\vee$ and $\supp(\sigma_r)^\vee=\langle e_{-k},\ldots, e_{-k+r-1}\rangle_\mathbb C\subset E_k^\vee$, the fiber of $\rho$ at the point $x_{\sigma_r,Q}:=\big(\sigma_r, [Q\cap Q^\perp \subset Q]\big)$ is
	\[ \rho^{-1}\big(x_{\sigma_r,Q}\big) = \left\{ \sigma_r + \sum_{i=1}^h v_i \otimes q_i + H' \ \bigg| \ v_i \in U, \ H' \in \supp(\sigma_r)^\vee \otimes E_k^\perp\!/E_k \right\}\]
	and it has dimension $\dim (U\otimes \mathbb C^h) + \dim\!\big( \supp(\sigma_r)^\vee \otimes E_k^\perp\!/E_k \big) = h(k-r) + r(2N-2k)$. From the fiber dimension theorem again, we deduce
	\begin{align*}
	\dim \cal O_{(r,h,t)} & = \dim \left[\Sym^2\!E_k^\vee \right]_r + \dim \cal I_{h,t} + \dim \rho^{-1}(x_{\sigma_r,Q})\\
	& = \frac{r(r+1)}{2} + (r+h)(2N-k-r) + \frac{t(t+1)}{2}-h^2-t^2 \ .
	\end{align*}
\end{proof}

\begin{es}\label{es:tangent space to IG(3,8)}
	Consider the $12$--dimensional isotropic Grassmannian $\IG(3,8)\subset \mathbb P(\bigwedge^3\mathbb C^8)\simeq \mathbb P^{56}$. From the constraints in \eqref{constraints on invariants}, the invariants $(r,h,t)$ are such that 
	\[ 0\leq r \leq 3 \ \ \ , \ \ \ h\leq \min\{2,3-r\} \ \ \ , \ \ \ t\leq \min\{h,1\} \ \ \ , \ \ \ t \equiv h \ \text{(mod $2$)} \ . \]
	In particular, in such a case, the value $t$ is uniquely determined by the value $h$: the first non trivial case in which $r,h,t$ are not redundant is $\IG(3,10)$. The tangent space $\mathfrak{sp}_8/\mathfrak{p}_3$ has the following poset of $P_3$--orbits, where the arrows denote the inclusion of an orbit into the closure of another.
	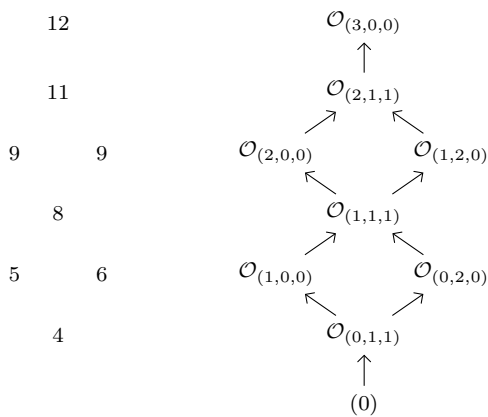
\begin{figure}[H]
		\begin{center}
			{\footnotesize \begin{tikzpicture}[scale=2.1]
				
				\node(000) at (0,-0.4){{$(0)$}};
				
				\node(011) at (0,0){{$\cal O_{(0,1,1)}$}};
				\node(011dim) at (-1.75,0){$4$};
				
				\node(100) at (-0.5,0.35){$\cal O_{(1,0,0)}$};
				\node(100dim) at (-2,0.35){$5$};
				
				\node(020) at (0.5,0.35){{$\cal O_{(0,2,0)}$}};
				\node(020dim) at (-1.5, 0.35){$6$};
				
				\node(111) at (0,0.7){{$\cal O_{(1,1,1)}$}};
				\node(111dim) at (-1.75,0.7){$8$};
				
				\node(200) at (-0.5,1.05){$\cal O_{(2,0,0)}$};
				\node(200dim) at (-2,1.05){$9$};
				
				\node(120) at (0.5,1.05){$\cal O_{(1,2,0)}$};
				\node(120dim) at (-1.5,1.05){$9$};
				
				\node(211) at (0,1.4){{$\cal O_{(2,1,1)}$}};
				\node(211dim) at (-1.75,1.4){$11$};
				
				\node(300) at (0,1.8){$\cal O_{(3,0,0)}$};
				\node(300dim) at (-1.75,1.8){$12$};
				
				\path[font=\scriptsize,>= angle 90]
				(000) edge [->] node [left] {} (011)
				(011) edge [->] node [left] {} (100)
				(111) edge [->] node [left] {} (200)
				(111) edge [->] node [left] {} (120)
				(011) edge [->] node [left] {} (020)
				(100) edge [->] node [left] {} (111)
				(200) edge [->] node [left] {} (211)
				(211) edge [->] node [left] {} (300)
				(020) edge [->] node [right] {} (111)
				(120) edge [->] node [right] {} (211);
				\end{tikzpicture}}
		\end{center}
		\caption{Poset graph of $P_3$--orbits in $\mathfrak{sp}_8/\mathfrak{p}_3$, and their dimensions.}
		\label{figure:graph IG}
	\end{figure}
\end{es}

\subsection{The tangential variety $\tau(\IG(k,2N))$}

We are now interested in recovering the orbit poset of the tangential variety to a non-cominuscule isotropic Grassmannian $\IG(k,2N)$, and eventually determining its (tangential-)identifiable locus. We keep the notation from the first part of this section. For simplicity we write $[E]=[\bold{e}_{[k]}]$ for the highest weight line $[E_k]\in \IG(k,2N)$. Each orbit $\cal O(r,h,t)$ in the tangent space $T_{[E]}\IG(k,2N)$ induces a $\Sp_{2N}$--orbit in the tangential variety, which we denote by 
\[ \Theta_{(r,h,t)} \ := \ \Sp_{2N}\cdot \left[\bold{e}_{[k]} + \sigma_r + H_{(r,h,t)}\right] \ . \]
We stress the fact that {\em a priori} two orbits in the tangent space may define the same orbit in the tangential variety. We show that this is the case indeed. \\
\indent From the isomorphism $E^\vee \simeq \bigwedge^{k-1}E \simeq \mathbb C^{2N}/(E^\perp)$, we rewrite the tangent space as
\begin{equation}\label{eq:tangent to IG inside Gr} 
T_{[E]}\IG(k,2N) \ = \ \Sym^2E^\vee \oplus \left( E^\vee \otimes \dfrac{E^{\perp}}{E} \right) \ \simeq \ \left( \bigwedge^{k-1}E \wedge \dfrac{\mathbb C^{2N}}{E^\perp} \right) \oplus \left( \bigwedge^{k-1}E\wedge \dfrac{E^\perp}{E} \right) \ . 
\end{equation}
The expression on the right-hand-side (RHS) has the advantage to identify $T_{[E]}\IG(k,N)$ inside the tangent space at $[E]$ to the Grassmannian $\Gr(k,N)$: in other words, it allows to look at $T_{[E]}\IG(k,2N)$ as the linear section of $T_{[E]}\Gr(k,2N)$ cutted by $\mathbb P(V_{\omega_k}^{\Sp_{2N}})$ (cf. Remark \ref{rmk:pullback O(1)}).\\
\indent The isomorphism $E^\vee \simeq \bigwedge^{k-1}E$ identifies $e_{-i}$ with $\bold{e}_{[k]\setminus \{i\}}:=e_{1}\wedge \ldots \wedge \widehat{e_{i}} \wedge \ldots \wedge e_{k}$ for any $i \in [k]$, where $\widehat{e_i}$ means that this entry is missing in the wedge product. In particular, for any $i \in [k]$ the quadratic form $e_{-i}^{2}\in \Sym^2E^\vee$ is identified with $\bold{e}_{[k]\setminus \{i\}}\wedge e_{-i}$, and therefore
\[ \sigma_r := \sum_{\ell \in [r]} e_{-k+\ell-1}^2 \ \simeq \ \sum_{\ell \in [r]}\bold{e}_{[k]\setminus \{k-\ell+1\}}\wedge e_{-k+\ell-1} \ , \]
\begin{align*}
H_{(h,t)} 
& := \ \sum_{i=1}^t e_{-i}\otimes e_{k+i} \ + \ \sum_{i=1}^{\frac{h-t}{2}} e_{-t-i}\otimes e_{k+t+i} \ + \ \sum_{i=1}^{\frac{h-t}{2}} e_{-t-\frac{h-t}{2}-i}\otimes e_{-k-t-i} \\
& \simeq \  \sum_{i=1}^t  \bold{e}_{[k]\setminus \{i\}}\wedge e_{k+i} \ + \ \sum_{i=1}^{\frac{h-t}{2}} \bold{e}_{[k]\setminus \{t+i\}} \wedge e_{k+t+i} \ + \ \sum_{i=1}^{\frac{h-t}{2}} \bold{e}_{[k]\setminus \{t+\frac{h-t}{2}+i\}}\wedge e_{-k-t-i} \ .
\end{align*}

\indent The above summands have two-by-two Hamming distance $2$, while each one of them has Hamming distance $1$ from $[E]=\bold{e}_{[k]}$. In particular,
\[ \Theta_{(1,0,0)} \ = \ \Theta_{(0,1,1)} \ = \ \IG(k,2N) \ \ \ \ \ , \ \ \ \ \ \Theta_{(r,h,t)} \neq \IG(k,2N) \ \ \ \forall \ r+h \geq 2 \ .\]
On the other hand, the representatives $[\bold{e}_{[k]} + \sigma_r + H_{(h,t)}]$ for $r+h=2$ are
\begin{align*}
    \bold{e}_{[k]}+ \bold{e}_{[k]\setminus \{1\}}\wedge e_{k+1} + \bold{e}_{[k]\setminus \{2\}}\wedge e_{-k-1} = \bold{e}_{[k]\setminus \{1\}}\wedge (e_{k+1} \pm e_1) + \bold{e}_{[k]\setminus \{2\}}\wedge e_{-k-1} & \in \Theta_{(0,2,0)} \\
    \bold{e}_{[k]}+ \bold{e}_{[k-1]}\wedge e_{-k} + \bold{e}_{[k]\setminus\{1\}}\wedge e_{k+1} = \bold{e}_{[k]\setminus \{1\}}\wedge (e_{k+1} \pm e_1) + \bold{e}_{[k-1]}\wedge e_{-k} & \in \Theta_{(1,1,1)} \\
    \bold{e}_{[k]}+ \bold{e}_{[k-1]}\wedge e_{-k} + \bold{e}_{[k]\setminus \{k-1\}}\wedge e_{-k+1} = \bold{e}_{[k]\setminus\{k-1\}}\wedge(e_{k-1} + e_{-k+1}) + \bold{e}_{[k-1]}\wedge e_{-k} & \in \Theta_{(2,0,0)} 
    \end{align*}
and they have rank$_{\IG(k,2N)}$ $2$, hence each of them lies on a non-degenerate (i.e. non-tangent) bisecant line too. Also, such points lie in the (linear section of the) distance-$2$ orbit $\Sigma_2^{\Gr(k,2N)}$ of the secant variety of lines to the Grassmannian $\Gr(k,2N)$ (cf. \cite[Remark 3.2.4]{galganostaffolani2022grass}). We conclude that 
\[ \tau(\IG(k,2N))\cap \sigma_2^\circ(\IG(k,2N)) \ = \  \Theta_{(2,0,0)} \sqcup \Theta_{(0,2,0)} \sqcup \Theta_{(1,1,1)} \ = \ \Sigma_2^{\Gr(k,2N)} \cap \mathbb P(V_{\omega_k}^{\Sp_{2N}}) \ . \] 

\noindent The above behaviour generalises to the other tangent orbits as follows.

\begin{prop} 
    Let $\Theta_{(r,h,t)}$ be the orbits in the tangential variety to the isotropic Grassmannian $\IG(k,2N)$. Let $\Theta_\ell^{\Gr(k,2N)}$ be the rank--$\ell$ orbit in the tangential variety to the Grassmannian $\Gr(k,2N)$ as defined in Corollary \ref{cor:tangent orbits}. Then the linear section equality $\tau(\IG(k,2N))=\tau(\Gr(k,2N))\cap \mathbb P(V_{\omega_k}^{\Sp_{2N}})$ induces the equality
    \[ \bigsqcup_{r+h=\ell}\Theta_{(r,h,t)} = \Theta_\ell^{\Gr(k,2N)}\cap \mathbb P(V_{\omega_k}^{\Sp_{2N}}) \ . \]
    \end{prop}

\begin{es}\label{es:tangential IG(3,8)}
We go back to Example \ref{es:tangent space to IG(3,8)}. Then \autoref{figure:graph tau IG} shows how the $P$--orbit poset graph for $T_{[E]}\IG(3,8)$ \textquotedblleft degenerate\textquotedblright \space to the $G$--orbit poset graph for $\tau(\IG(3,8))$. Violet, red and green orbits arise from linear sections of $\Gr(3,8), \Sigma_2^{\Gr(3,8)}, \Theta_3^{\Gr(3,8)}$ respectively.

\begin{figure}[h]
\begin{center}
		\begin{minipage}[l]{0.49\textwidth}
        \centering
			{\footnotesize \begin{tikzpicture}[scale=2]
				
				\node(000) at (0,-0.4){\textcolor{violet}{$(0)$}};
				
				\node(011) at (0,0){\textcolor{violet}{$\cal O_{(0,1,1)}$}};
				
				\node(100) at (-0.5,0.35){\textcolor{violet}{$\cal O_{(1,0,0)}$}};
				
				\node(020) at (0.5,0.35){\textcolor{red}{$\cal O_{(0,2,0)}$}};
				
				\node(111) at (0,0.7){\textcolor{red}{$\cal O_{(1,1,1)}$}};
				
				\node(200) at (-0.5,1.05){\textcolor{red}{$\cal O_{(2,0,0)}$}};
				
				\node(120) at (0.5,1.05){\textcolor{green}{$\cal O_{(1,2,0)}$}};
				
				\node(211) at (0,1.4){\textcolor{green}{$\cal O_{(2,1,1)}$}};
				
				\node(300) at (0,1.8){\textcolor{green}{$\cal O_{(3,0,0)}$}};
				
				\path[font=\scriptsize,>= angle 90]
				(000) edge [->] node [left] {} (011)
				(011) edge [->] node [left] {} (100)
				(111) edge [->] node [left] {} (200)
				(111) edge [->] node [left] {} (120)
				(011) edge [->] node [left] {} (020)
				(100) edge [->] node [left] {} (111)
				(200) edge [->] node [left] {} (211)
				(211) edge [->] node [left] {} (300)
				(020) edge [->] node [right] {} (111)
				(120) edge [->] node [right] {} (211);
				\end{tikzpicture}}
		\end{minipage}
    \hfill
		\begin{minipage}[r]{0.49\textwidth}
        \centering
			{\footnotesize \begin{tikzpicture}[scale=2.3]
				
                
				\node(000) at (0,0){\textcolor{violet}{$\IG(k,2N)$}};

                \node(020) at (0,0.35){\textcolor{red}{$\Theta_{(0,2,0)}$}};
				
				\node(111) at (0,0.7){\textcolor{red}{$\Theta_{(1,1,1)}$}};
				
				\node(200) at (-0.5,1.05){\textcolor{red}{$\Theta_{(2,0,0)}$}};
				
				\node(120) at (0.5,1.05){\textcolor{green}{$\Theta_{(1,2,0)}$}};
				
				\node(211) at (0,1.4){\textcolor{green}{$\Theta_{(2,1,1)}$}};
				
				\node(300) at (0,1.8){\textcolor{green}{$\Theta_{(3,0,0)}$}};
				
				\path[font=\scriptsize,>= angle 90]
				(000) edge [->] node [left] {} (020)
				(020) edge [->] node [left] {} (111)
				(111) edge [->] node [left] {} (120)
				(111) edge [->] node [left] {} (200)
				(200) edge [->] node [left] {} (211)
				(211) edge [->] node [left] {} (300)
				(020) edge [->] node [right] {} (111)
				(120) edge [->] node [right] {} (211);
				\end{tikzpicture}}
		\end{minipage}
		\caption{Poset graph of $\Sp_8$--orbits in $\tau(\IG(3,8))$.}
		\label{figure:graph tau IG}
  \end{center}
	\end{figure}
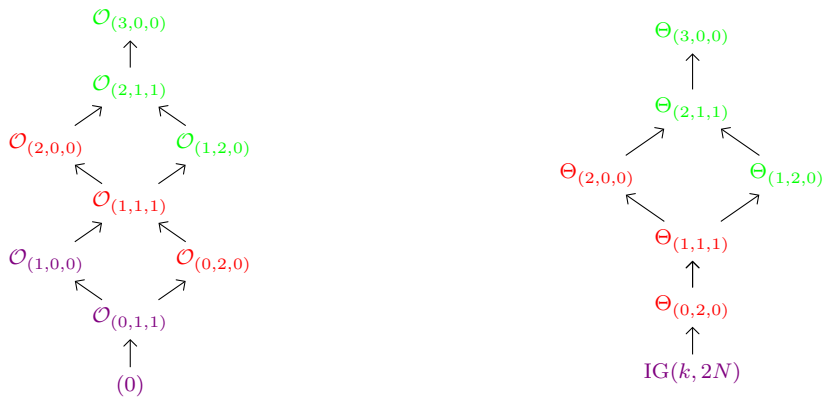
\end{es}

\noindent We recall that a point in the tangential variety is tangential-identifiable if it lies on a unique tangent space.

\begin{prop}
    The tangential-identifiable locus of the tangential variety $\tau(\IG(k,2N))$ coincides with the union of all tangent orbits $\Theta_{(r,h,t)}$ such that $r+h\geq 3$.
\end{prop}
\begin{proof}
 For any $r+h=\ell \geq 3$, the tangential-identifiability of points in the orbit $\Theta_{(r,h,t)}$ follows from the tangential-identifiability of the orbits $\Theta_{\ell}^{\Gr(k,2N)}$ (cf. \cite[Sec.\ 5]{galganostaffolani2022grass}). For $r+h=2$, the representatives of the orbits $\Theta_{(0,2,0)},\Theta_{(1,1,1)}, \Theta_{(2,0,0)}$ lie on the tangent spaces at $e_2\wedge \ldots \wedge e_k \wedge e_{-k-1}$, $(e_{k+1}\pm e_1)\wedge e_2\wedge \ldots \wedge e_{k-1}\wedge e_{-k}$ and $e_1\wedge \ldots \wedge e_{k-2}\wedge (e_{k-1} + e_{-k+1})\wedge e_{-k}$ respectively, besides lying on the tangent space at $\bold{e}_{[k]}$.
\end{proof}

{\tiny
\printbibliography
}

\end{document}